\documentclass[10pt,a4paper]{article}
\usepackage{amsopn}

\usepackage{graphics}
\usepackage{graphicx}
\usepackage{epsfig}

\usepackage{amssymb}
\usepackage{amsthm}
\usepackage{amsmath}
\usepackage{amsfonts}
\usepackage{lipsum}
\usepackage[left=2.5cm,right=2.5cm,bottom=3.5cm,top=3.5cm]{geometry}
\usepackage{empheq}
\usepackage{epstopdf}
\usepackage{changepage}
\usepackage{rotating}
\usepackage{adjustbox}
\usepackage{graphbox}
\usepackage{color}
\usepackage{xcolor}
\usepackage{dsfont}
\usepackage{dutchcal}
\usepackage{hyperref} 
\usepackage{float}
\usepackage{resizegather}
\usepackage{multicol}
\usepackage{booktabs} 
\usepackage{colortbl} 
\usepackage{multirow} 
\usepackage{longtable,tabu} 
\usepackage{hhline}   
\usepackage{anyfontsize} 
\usepackage{everysel}
\usepackage{psfrag}
\usepackage{titlesec}
\usepackage{bbm}
\usepackage{bm}
\usepackage{pgf,tikz}

\newtheorem{lema}[section]{Lemma}

\newtheorem{remark2}[]{Remark}

\def\CFL{\rm{CFL}}
\newcommand{\dep}[2]{{\displaystyle \frac{\partial #1}{\partial #2}}}

\allowdisplaybreaks
\hypersetup{
	pdftitle={A projection hybrid high order finite volume/finite element method for incompressible turbulent flows},    
	pdfauthor={S. Busto, J.L. Ferrin, E.F. Toro, M.E. Vazquez-Cendon},     
	pdfcreator={S. Busto, J.L. Ferrin, E.F. Toro, M.E. Vazquez-Cendon},   
	colorlinks=true,       
}

\title{A projection hybrid high order finite volume/finite element method for incompressible turbulent flows} 

\begin{document}

\begin{center}
	\textbf{ \Large{A semi-implicit hybrid finite volume / finite element scheme \\ for all Mach number flows on staggered unstructured meshes} }
	
	\vspace{0.5cm}
	{S. Busto\footnote{saray.busto@usc.es}, J. L. Ferr\'in\footnote{joseluis.ferrin@usc.es}, E. F. Toro\footnote{eleuterio.toro@unitn.it}, M. E. V\'azquez-Cend\'on\footnote{elena.vazquez.cendon@usc.es}}
	
	\vspace{0.2cm}
	{\small
		\textit{$^{(1,2,4)}$ Departamento de Matem\'atica Aplicada, Universidade de Santiago de Compostela. Facultad de Matem\'aticas, ES-15782 Santiago de Compostela, Spain}
		
		\textit{$^{(3)}$ Department of Civil, Environmental and Mechanical Engineering, University of Trento, Via Mesiano 77, 38123 Trento, Italy}
	}
\end{center}

\hrule
\vspace{0.4cm}

\begin{center}
	\textbf{Abstract}
\end{center} 

\vspace{0.1cm}
In this paper the  projection hybrid FV/FE method presented in \cite{BFSV14} is extended to account for species transport equations.
Furthermore, turbulent regimes are also considered thanks to the $k-\varepsilon$ model.
Regarding the transport diffusion stage new schemes of high order of accuracy are developed. The CVC Kolgan-type scheme and ADER methodology are extended to 3D. The latter is modified in order to profit from the dual mesh employed by the projection algorithm and the derivatives involved in the diffusion term are discretized using a Galerkin approach.
The accuracy and stability analysis of the new method are carried out for the advection-diffusion-reaction equation.
Within the projection stage the pressure correction is computed by a piecewise linear finite element method.
Numerical results are presented, aimed at verifying the formal order of accuracy of the scheme and to assess the performance of the method on several realistic test problems.
\
\vspace{0.2cm}

\vspace{0.4cm}

\hrule

\section{Introduction}
Finite volume methods combined with approximate
Riemann solvers have been successfully developed for different kinds of
flows in the 1980's (see, \cite{Toro} and the references
therein). Focusing on the incompressible case, pressure results in a
Lagrange multiplier that adapts itself to ensure that the velocity
satisfies the incompressibility condition. In order to handle
this situation, the typical explicit stage of finite volume
methods has to be complemented with the so-called projection stage
where a pressure correction is computed in order to get a
divergence-free velocity. Many papers exist in the literature
devoted to introduce and analyse projection finite volume methods
for incompressible Navier-Stokes equations (see, for instance,
\cite{BDALR02} or \cite{PABCHL95}). In order to get stability, staggered
grids have been used to discretize velocity and pressure. While
this can be done straightforwardly in the context of structured
meshes, the adaptation to unstructured meshes is more challenging (see
\cite{BDDV98}, \cite{GGHL08}, \cite{GLL12},
\cite{PBH04}, \cite{TA07}, \cite{TAPW09}).

On the other hand, projection methods have been also used in
combination with finite element discretizations (see
\cite{Guer06}). Within this approach, the divergence-free
condition for the velocity is replaced by an equation prescribing the divergence of
the linear momentum density which is a conservative variable.

The scope of this paper is to extend the hybrid FV/FE projection method
introduced in \cite{BFSV14} for both laminar and turbulent flows considering also
transport of species.
Furthermore, new methods to increase the accuracy of the methodology are developed.

Starting from a 3D tetrahedral finite element mesh of the
computational domain, the equation of the transport-diffusion stage
is discretized by a finite volume method associated with a dual
finite volume mesh where the nodes of the volumes are the barycentre
of the faces of the initial tetrahedra. These volumes, which allow us
for an easy implementation of flux boundary conditions, have already
been used, among others, for the 2D shallow water equation (see \cite{BDDV98}),
for solving conservative and non conservative systems in 2D and 3D (see \cite{THD09} and  \cite{DHCPT10}) 
and for DG schemes employed to solve compressible Navier-Stokes equations (see \cite{TD17}).
For time discretization we use the explicit Euler
scheme. The convective term is upwinded using the Rusanov scheme (see
\cite{Toro} and \cite{VC15}). Concerning the projection stage, the
pressure correction is computed by continuous piecewise linear
finite elements associated with the initial tetrahedral mesh. The
use of the above ``staggered'' meshes together with a simple
specific way of passing the information from the transport-diffusion
stage to the projection one and vice versa leads to a stable scheme.
The former is done by redefining the conservative variable (i.e. the
momentum density) constant per tetrahedron. Conversely, the finite
element pressure correction is redefined to constant values on the
faces of the finite volumes and then used in the transport-diffusion
stage. 

The coupling of Navier-Stokes equations and the turbulence model 
introduces turbulent viscosity which is typically computed by solving
an additional pair of advection-diffusion-reaction equations, that is
equations for the turbulent kinetic energy and { the} dissipation rate. 
One issue here is the time dependency of the viscous terms. 
This requires the use of methods that are at least second-order accurate
in space and time for all terms involved (see \cite{CFVC06} and \cite{Saa11}).

For advection equations, several approaches for constructing high-order
methods have been put forward. A classical example is the Lax-Wendroff 
scheme (see \cite{LW60} and \cite{Lax57}). This scheme is linear in the
sense of Godunov and thus oscillatory, according to Godunov's theorem, \cite{God59}.
A major step forward in this direction was the work of Kolgan \cite{Kol11},
who introduced, for the first time, a numerical scheme that circumvents Godunov's theorem,
via the construction of a non-linear scheme using non-linear reconstructions (see \cite{Ber06}).
{ Following these works, a new Kolgan-type method
has been presented for the shallow water equations in \cite{CV12}. In what follows,
we will refer to this scheme as the CVC Kolgan-type scheme.}

In the present paper, the CVC Kolgan-type scheme is analysed and implemented at
the transport-diffusion stage for the convective
terms of the considered conservation laws: momentum conservation,
transport equations and  $k$-$\varepsilon$ model. The obtained scheme,
second order in space and first order in time 
is combined with a Galerkin approach of the gradients involved in the viscous term.
An alternative option will be the decomposition of the diffusion term into 
its orthogonal and non-orthogonal parts as introduced in \cite{BFSV14}.

More advanced non-linear methods for advection dominated
problems have appeared in the literature since the introduction of Kolgan scheme.
Some of them are: Total Variation Diminishing Methods (TVD), Flux Limiter Methods,
MUSCL-Hancock, semi-discrete ENO or WENO (see, for instance, \cite{Leer84},
\cite{VL97}, \cite{Swe84}, \cite{Leer84}, \cite{HEOC87}, \cite{CT08},
\cite{SO88} and \cite{LSC94}). Comprehensive reviews are found in
\cite{Toro} and \cite{LV02}, for example.
Focusing on high order in time and space methodologies, we highlight the ADER approach,
first put forward in \cite{TMN01}. It is  also a fully discrete approach that relies
on non-linear reconstructions and the solution of the generalised Riemann problem,
to any order of accuracy. The resulting schemes are arbitrarily accurate in both
space and time in the sense that they have no theoretical accuracy barrier.
An introduction to ADER schemes is found in Chapters 19 and 20 of \cite{Toro}. 
Further developments and applications are found, for example, in {
\cite{TT04}, \cite{DM05}, \cite{Tit05}, \cite{TT05tvd}, \cite{TT05}, \cite{TDTK06}, \cite{TT06}, \cite{TT06ecc},
\cite{Tak06}, \cite{TT07}, \cite{Zah08}, \cite{DBTM08}, \cite{DET08},
 \cite{Dum10},  \cite{HD11}, \cite{BD14}, \cite{MT14}.}

In \cite{BTVC16} an extension of ADER methodology to solve the advection-diffusion-reaction equation,
admitting space and time dependent diffusion coefficients was introduced.
The present work includes a modification of this scheme, the Local ADER method (LADER),
which profits from the dual mesh. Moreover, an ENO-based reconstruction is
considered in order to prevent spurious oscillations.

To asses the performance of the methodology
different manufactured solution tests are introduced and the 
numerical results obtained with the developed computer
code are shown.
Furthermore, several classical test problems from fluid mechanics are presented
and some results are compared with experimental data (see \cite{Saa11} and \cite{BFSV14}).

The paper is organized as follows.
In Section \ref{sec:math} the mathematical model for incompressible flows is recalled.
Then, the RANS $k-\varepsilon$ model for the turbulence and the species transport equations are described. 
In Section \ref{sec:ndfv} the numerical discretization is detailed.
Special attention is paid to the description of the finite volume algorithm.
Aiming to achieve a high order scheme, two different methodologies for the
flux terms are developed: the CVC Kolgan-type method, second order in space
and first order in time, and the LADER methodology, second order in space and time.
The needed modifications on the approximation of remaining terms of the equations
to achieve a high order scheme are also presented and a Galerkin approach
to compute the diffusion terms is introduced.
Finally, in Section \ref{sec:numer_res} some of the numerical results
obtained with the developed code are shown. On the one hand, the order
of convergence of the method is analysed using the method of manufactured solutions.
On the other hand, several classical test problems are analysed.
The appendix includes the theoretical analysis of the LADER numerical
method applied to the one-dimensional advection-diffusion-reaction equation.

\section{Governing equations} \label{sec:math}
In this section, the system of equations to be solved is introduced.
The model for 
incompressible newtonian fluids, recalled in \cite{BFSV14}, is extended considering
a turbulent regime and taking into account the transport of species.

\subsection{Mass conservation and momentum equations}
The incompressible Navier-Stokes equations reduce to the
mass conservation equation and the momentum equation.
Hence, the system of equations written in conservative variables reads
\begin{align}
&& \mathrm{div}\mathbf{w}_\mathbf{u}  =0, \label{eq:edcg}\\
&& \dep{\mathbf{w}_\mathbf{u}}{t} +{ { \mathrm{div}  \mathbf{{\cal
			F}^{\mathbf{w}_\mathbf{u}}}(\mathbf{w}_{\mathbf{u}}} )} +{ \nabla \pi }
-\mathrm{div}(\tau)={ \mathbf{f}_{\mathbf{u}} }. \label{eq:emcg}
\end{align}
Standard notation is used: 
\begin{itemize}
	\item $\rho$ is the density (kg/m$^3$),
	\item $p=\pi + \overline{\pi}$ is the pressure (N/m$^2$),
	\begin{itemize}
		\item $\overline{\pi}$ is the mean pressure,
		\item $\pi$ is the pressure perturbation,
	\end{itemize}
	\item $\mathbf{u}=(u_1,u_2,u_3)^t$ is the velocity vector (m/s),
	\item $\mathbf{w}_{\mathbf{u}}:=\rho\mathbf{u}$ is the vector of the
	conservative variables related to the velocity (kg/s$\:$m$^2$),
	\item $\mathbf{{\cal F}}^{\mathbf{w_u}}$ is the flux tensor:
	\[{ \mathbf{\cal F}_{i}^{\mathbf{w}_{\mathbf{u}}}
	(\mathbf{w}_{\mathbf{u}})= \frac{1}{\rho} \mathbf{w}_i \mathbf{w}_{\mathbf{u}}  }=
	u_i \mathbf{w}_{\mathbf{u}},\; i=1,2,3,\]
	\item  $\tau$ is the viscous part of the Cauchy stress tensor, 
	\item $\mathbf{f}_{\mathbf{u}}$ is a generic source term  used for the manufactured test problems.
\end{itemize}

\subsection{Turbulence model}
Special care for the viscous part of Cauchy stress tensor, $\tau $, is required for turbulent regimes.
In order to avoid the high computational cost of a direct simulation of the turbulence, the $k-\varepsilon$
standard model is used (see  \cite{Ber05} and \cite{Cha14}). The Reynolds-averaged viscous stress tensor is given by
\begin{equation}
\tau = \tau_{\mathbf{u}} + \tau^{\mathbf{R}}.\end{equation}
Denoting $\mu$ the laminar viscosity (kg/(m s)),  the averaged stress tensor, $\tau_{\mathbf{u}}$, reads
\begin{equation}\tau_{\mathbf{u}}=\mu \left(\nabla \mathbf{u}+\nabla\mathbf{u}^T\right).\end{equation}
The fluctuation, $\tau^{\mathbf{R}}$, called the Reynolds tensor, is given by
\begin{equation}\tau^{\mathbf{R}}=\mu_t \left(\nabla \mathbf{u}+\nabla\mathbf{u}^T\right)-\frac{2}{3} \rho k \mathbf{I}.\end{equation} 
To obtain the turbulent viscosity,
\begin{equation}\mu_t=\rho C_\mu \frac{k^2}{\varepsilon},\end{equation}
two new variables are introduced:   
the turbulent kinetic energy, $k$ (J/kg), and  the energy dissipation rate, $\varepsilon$ (J/(kg s)).
They are computed from a new pair of partial differential equations, namely,
\begin{eqnarray}
\frac{\partial w_{k}}{\partial t}+
\mathrm{div} \mathcal{F}^{w_k}\left(w_{k},{ \mathbf{u}}\right)
-\mathrm{div}\left[\left(\mu+\frac{\mu_t}{\sigma_k}\right)\nabla\left( \frac{w_k}{\rho}\right) \right]
+w_{\varepsilon}=G_k+f_{k},\label{eq:tasa_disipacion_viscosa1} \\[0.3cm] 
\frac{\partial w_{\varepsilon}}{\partial t}+
\mathrm{div} \mathcal{F}^{w_\varepsilon}\left(w_{\varepsilon},{ \mathbf{u}}\right)
-\mathrm{div}\left[\left(\mu+\frac{\mu_t}{\sigma_\varepsilon}\right)\nabla \left( \frac{w_{\varepsilon}}{\rho}\right) \right]
+ C_{2\varepsilon} \frac{w_\varepsilon^2}{w_{k}} =C_{1\varepsilon} \frac{w_{\varepsilon}}{w_{k}}G_k+f_{\varepsilon},
\label{eq:energia_cinetica_turbulenta1}
\end{eqnarray}
where
\begin{itemize}
	\item $w_{k}$ (J), $w_{\varepsilon}$ (J/s) are the conservative variables corresponding to $k$ and $\varepsilon$, that is
	\begin{equation*}
	w_{k}:=\rho k, \quad 	w_{\varepsilon}:=\rho\varepsilon,
	\end{equation*}
	\item  $\mathcal{F}^{w_{k}} ,\, \mathcal{F}^{w_\varepsilon} $ are the fluxes related to the turbulence variables,
	\begin{equation*}
	\mathcal{F}^{w_{k}}_i\left(w_{k},{ \mathbf{u}}\right)=u_i w_{k}, \quad
	\mathcal{F}^{w_{\varepsilon}}_i\left(w_{\varepsilon},{ \mathbf{u}}\right)=u_i w_{\varepsilon},
	\end{equation*}
	\item $G_k$ is the term of kinetic energy production, due to the mean velocity gradients, of the Reynolds stress tensor,
	\begin{equation}
	G_k=\frac{\mu_t}{2} \left[\sum_{i=1}^{3}\sum_{j=1}^{3}\left( \frac{\partial u_i}{\partial x_j}
	+\frac{\partial u_j}{\partial x_i}\right) \right]^2,\label{eq:gk}
	\end{equation}
	\item $f_{k}$, $f_{\varepsilon}$ are the source terms related to manufactured
	solutions for analytical tests; they have zero value in physical problems,
	\item $\sigma_k=1.0$, $\sigma_\varepsilon=1.3$ are the turbulent Prandtl numbers,
	\item $C_{\mu}=0.09$, $C_{1\varepsilon}=1.44$, $C_{2\varepsilon}=1.92$ are the closure
	coefficients of the model whose values were taken from the literature. 
\end{itemize}

\subsection{Species transport}
The equations of transport of species are also included in the system to be solved:
\begin{eqnarray}
\frac{\partial \mathbf{w}_{\mathbf{y}}}{\partial t}+\mathrm{div}
\mathcal{F}^{\mathbf{w}_{\mathbf{y}}}\left(\mathbf{w}_{\mathbf{y}},{ \mathbf{u}}\right) 
-\mathrm{div}\left[\left(\rho \mathcal{D}+\frac{\mu_t}{Sc_t}\right)\nabla\left( \frac{1}{\rho}\mathbf{w}_{\mathbf{y}}\right) \right]
=\mathbf{f}_{\mathbf{y}},\label{eq:especies1}
\end{eqnarray}
with
\begin{itemize}
	\item $ \mathbf{y}=\left(\mathrm{y}_{1},\dots,\mathrm{y}_{N_{e}}\right)^{T}$ 
	the mass fraction vector of the species to be considered. $\mathrm{y}_i$ corresponds
	to { species} $i$ and $N_e$ is the total number of species to be considered,
	\item $\mathbf{w}_{\mathbf{y}}:=\rho\mathbf{y}$  the conservative variable vector related to the mass fraction vector,
	\item $\mathcal{F}^{\mathbf{w}_{\mathbf{y}}}$  the flux,
	\begin{equation*}
	\mathcal{F}^{\mathbf{w}_{\mathbf{y}}}_i\left(\mathbf{w}_{\mathbf{y}},{ \mathbf{u}}\right)=u_i \mathbf{w}_{\mathbf{y}}
	\end{equation*}
	\item $\mathcal{D}$  the mass diffusivity coefficient (m$^2$/s),
	\item $Sc_t=0.7$  the turbulent Schmidt number,
	\item $\mathbf{f}_{\mathbf{y}}$  the source term for manufactured test problems.
\end{itemize}

\subsection{Complete system}
The complete system of equations to be solved is
\begin{eqnarray}
\mathrm{div}\mathbf{w}_\mathbf{u}  =0, \label{eq:masa}\\
\dep{\mathbf{w}_\mathbf{u} }{t} + { \textrm{div} \mathbf{{\cal
			F}^{w_u}}(\mathbf{w}_\mathbf{u} )} +{ \nabla \pi }
-\mathrm{div}(\tau)={ \mathbf{f}_{\mathbf{u}} }, \label{eq:momentos} \\
\frac{\partial w_{k}}{\partial t}+
\mathrm{div} \mathcal{F}^{w_k}\left(w_{k},\mathbf{u}\right)
-\mathrm{div}\left[\left(\mu+\frac{\mu_t}{\sigma_k}\right)\nabla\left( \frac{w_k}{\rho}\right) \right]+w_{\varepsilon}=G_k+f_{k}, \label{eq:tasa_disipacion_viscosa2}\\  
\frac{\partial w_{\varepsilon}}{\partial t}\!+\!
\mathrm{div} \mathcal{F}^{w_\varepsilon}\!\left(w_{\varepsilon},\mathbf{u}\right)\!
-\!\mathrm{div}\!\left[\!\left(\mu+\frac{\mu_t}{\sigma_\varepsilon}\right)\nabla \left( \frac{w_{\varepsilon}}{\rho}\right)\! \right]\!\!+\! C_{2\varepsilon} \frac{w_\varepsilon^2}{w_{k}}\!=\!C_{1\varepsilon} \frac{w_{\varepsilon}}{w_{k}}G_k\!+\!f_{\varepsilon}, \label{eq:energia_cinetica_turbulenta2}\\ 
\frac{\partial \mathbf{w}_{\mathbf{y}}}{\partial t}+\mathrm{div} \mathcal{F}^{\mathbf{w}_{\mathbf{y}}}\left(\mathbf{w}_{\mathbf{y}},\mathbf{u}\right) 
-\mathrm{div}\left[\left(\rho \mathcal{D}+\frac{\mu_t}{Sc_t}\right)\nabla\left( \frac{1}{\rho}\mathbf{w}_{\mathbf{y}}\right) \right]=\mathbf{f}_{\mathbf{y}}.\label{eq:especies2} 
\end{eqnarray}
Moreover, the vector of the conservative variables is
$\mathbf{w}=(\mathbf{w}_{\mathbf{u}},\widehat{\mathbf{w}} )^T, $
with $\widehat{\mathbf{w}}$ the vector of the conservative variables
related with turbulence, { and species}, i.e.,
$\widehat{\mathbf{w}} =(w_{k}, w_{\varepsilon}, \mathbf{w}_{\mathbf{y}} )^T$.
The flux tensor of the complete system has three components:
\begin{equation}
\mathcal{F}=\left( \mathcal{F}_1 | \mathcal{F}_2 | \mathcal{F}_3 \right)_{\left( 3+2+N_e\right) \times 3}, \quad \mathcal{F}_i\left(\mathbf{w}\right) =
\frac{w_i}{\rho} \mathbf{w},\quad i=1,2,3.
\end{equation}

\section{Numerical discretization  \label{sec:ndfv}}
The numerical discretization of the complete system is performed by extending
the projection method first put forward in \cite{BFSV14}.
The developed numerical method solves, at each time step, equations
\eqref{eq:momentos}-\eqref{eq:especies2} with a finite volume method (FVM) and, 
so, an approximation of $\mathbf{w}$ is obtained. The next step is 
applying projection to system \eqref{eq:masa}-\eqref{eq:momentos}.
The pressure correction is provided by a piecewise
linear finite element method (FEM). In the post-projection step, an approximation of
$\mathbf{w}_{\mathbf{u}}$ verifying the divergence condition, \eqref{eq:masa}, is
obtained. Furthermore, the production terms of the turbulence equations
are also computed in this step to account for the corrected velocities.
The reaction terms are treated via a semi-implicit method.

We start by considering a two-stage in time discretization
scheme:
in order to get the solution
at time $t^{n+1}$, we use the previously obtained approximations
$\mathbf{W}^n$ of the conservative variables $\mathbf{w}(x,y,z,t^n)$,
$\mathbf{U}^n$ of the velocity $\mathbf{u}(x,y,z,t^n)$
and $\pi^n$ of the pressure perturbation $\pi(x,y,z,t^n)$, and
compute $\mathbf{W}^{n+1}$ and $\pi^{n+1}$ from the following system
of equations:
\begin{eqnarray}
\frac{1 }{\Delta t}\left( \widetilde{\mathbf{W}}^{n+1}_{\mathbf{u}}-\mathbf{W}^{n}_{\mathbf{u}}\right)   +  \mathrm{div} \mathbf{{\cal F}}^{\mathbf{w}_{\mathbf{u}}}(
\mathbf{W}^{n}_{\mathbf{u}})
+  \nabla \pi^{n} - \mathrm{div}( \tau^n)=\mathbf{f}^n_{\mathbf{u}} , \label{eq:tidenincrec} \\
\frac{ 1 }{\Delta t} \left( \mathbf{W}^{n+1}_{\mathbf{u}}-  \mathbf{\widetilde{W}}^{n+1}_{\mathbf{u}}\right)  + 
\nabla (  \pi^{n+1} - \pi^{n}) =0, \,   \label{eq:tideincre2} \\
\mathrm{div}\mathbf{W}^{n+1}_{\mathbf{u}}=0, \label{eq:qincrec}\\[0.4cm]
\frac{1}{\Delta t}\left( \widetilde{W}_k^{n+1}-W_k^n\right) +\mathrm{div}\mathcal{F}^{w_{k}}\left(W_{k}^n,\mathbf{U}^{n} \right) -\mathrm{div}\left[ \left( \mu+\frac{\mu_t}{\sigma_k} \right)\nabla \frac{W_k^n}{\rho} \right]=0,\label{eq:k_discret}\\
\frac{1}{\Delta t}\left( W_k^{n+1}-\widetilde{W}_k^{n+1}\right) + W_{\varepsilon}^{n}-G_k^{n+1}=f^{n}_k,\label{eq:k_post_disc}\\[0.3cm]
\frac{1}{\Delta t}\left( \widetilde{W}_{\varepsilon}^{n+1}-W_{\varepsilon}^n\right) +\mathrm{div}\mathcal{F}^{w_{\varepsilon}}\left(W_{\varepsilon}^n,\mathbf{U}^{n} \right) -\mathrm{div}\left[ \left( \mu+\frac{\mu_t^{n}}{\sigma_{\varepsilon}} \right)\nabla\frac{W_{\varepsilon}^n}{\rho} \right]=0,\label{eq:epsilon_discret}\\
\frac{1}{\Delta t}\left( W_{\varepsilon}^{n+1}-\widetilde{W}_{\varepsilon}^{n+1}\right) + C_{2\varepsilon}\frac{W_{\varepsilon}^{n+1}W_{\varepsilon}^{n}}{W_{k}^{n}} 
-C_{1\varepsilon}\frac{W_{\varepsilon}^{n}}{W_{k}^{n}}G_k^{n+1}=f^{n}_{\varepsilon},\label{eq:epsilon_post_disc}\\
\frac{1}{\Delta t}\!\left( \widetilde{\mathbf{W}}_{\mathbf{y}}^{n+1}\!-\!\mathbf{W}_{\mathbf{y}}^n\right) \!+\!\mathrm{div}\mathcal{F}^{\mathbf{w}_{\mathbf{y}}}\!\left(\mathbf{W}_{\mathbf{y}}^n,\mathbf{U}^{n} \right)\! -\!\mathrm{div}\left[\! \left( \rho\mathcal{D}+\frac{\mu_t^{n}}{Sc_t} \right)\nabla\! \left(\!\frac{1}{\rho}\mathbf{W}_{\mathbf{y}}^n\!\right) \!\right] = 0,\label{eq:especies_discret}\\
\frac{1}{\Delta t}\left( \mathbf{W}_{\mathbf{y}}^{n+1} - \widetilde{\mathbf{W}}_{\mathbf{y}}^{n+1}\right)  =  \mathbf{f}^{n}_{\mathbf{y}}. \label{eq:especies_post_disc}
\end{eqnarray}
Concerning the discretization of mass conservation and momentum equations,
by adding equations \eqref{eq:tidenincrec}-\eqref{eq:tideincre2}, we easily see that the scheme is
actually implicit for the pressure term. However, the writing above
shows that the pressure and the velocity can be solved in three uncoupled stages. The first of them
corresponds to equation \eqref{eq:tidenincrec} and will be called
the transport-diffusion stage; it is explicit and allows us to
compute the intermediate approximation of the conservative variables
$\widetilde{\mathbf{W}}^{n+1}_{\mathbf{u}}$ (we notice that, in general, this
approximation  does not satisfy the divergence condition
\eqref{eq:qincrec}). The second one, to be called the projection
stage, is implicit; it consists of solving the coupled equations
\eqref{eq:tideincre2} and \eqref{eq:qincrec} with a finite element method
to obtain the pressure correction $\delta^{n+1}:=\pi^{n+1}-\pi^{n}$.
The last one is the post-projection stage; the intermediate approximation for the
velocity conservative variables is updated with the pressure correction providing
the final approximations $\mathbf{W}^{n+1}_{\mathbf{u}}$ and $\pi^{n+1}$ 
(see \cite{BFSV14} for further details).

As a novelty in this paper, for the remaining conservation laws 
the approximation of the conservative variables is obtained in two steps.
At the transport-diffusion stage we compute an approximation of the conservative variables,
$\widetilde{W}_k^{n+1}$, $\widetilde{W}_{\varepsilon}^{n+1}$ and $\widetilde{\mathbf{W}}_{\mathbf{y}}^{n+1}$,  
taking into account the corresponding flux and diffusion terms. Let us remark that at
this stage the update of the approximations involves all the neighbouring nodes of the
finite volume $C_{i}$. On the other hand, the discretization of the source terms related
to manufactured solutions or other production terms in equations  \eqref{eq:k_post_disc},
\eqref{eq:epsilon_post_disc} and \eqref{eq:especies_post_disc} involves pointwise evaluations
at the cell $C_{i}$ which will be computed at the post-projection stage. Furthermore, the
production terms of the turbulence equations are computed taking into account the updated
velocities and the reaction terms are treated via a semi-implicit method. As a result,
we obtain the updated conservative variables $\widetilde{W}_k^{n+1}$, $\widetilde{W}_{\varepsilon}^{n+1}$
and $\widetilde{\mathbf{W}}_{\mathbf{y}}^{n+1}$. 

Summarizing, the overall method consists of:
\begin{itemize}
	\item {\it Transport-diffusion stage:} equations  \eqref{eq:tidenincrec},
	\eqref{eq:k_discret}, \eqref{eq:epsilon_discret} and \eqref{eq:especies_discret}
	are solved by a FVM. 
	\item {\it Projection stage:} the pressure correction
	$\delta^{n+1}:=\pi^{n+1}-\pi^n$ is obtained by solving equations \eqref{eq:tideincre2}
	and (\ref{eq:qincrec}) with a FEM.
	\item {\it Post-projection stage:}  the $\mathbf{\widetilde{W}}^{n+1}_{\mathbf{u}}$ computed
	at the first stage 	is updated by using $\delta^{n+1}$ in order to obtain another
	approximation $\mathbf{W}^{n+1}_{\mathbf{u}}$, satisfying  the divergence
	condition (\ref{eq:qincrec}). Next, the turbulence and species
	variables are updated from equations  \eqref{eq:k_post_disc},
	\eqref{eq:epsilon_post_disc} and \eqref{eq:especies_post_disc}, 
	respectively.
\end{itemize}

\subsection{A dual finite volume mesh}
For the space discretization we consider a 3D unstructured tetrahedral
finite element mesh $\{ T_k, \, i=1,\dots ,nel\}$. From this mesh we
build a {\it dual finite volume mesh} as introduced in \cite{BFSV14} and \cite{BDDV98}. The
nodes, to be denoted by $\{ N_i, \, i=1,\dots ,nvol\}$, are the
barycenters of the faces of the initial tetrahedra.  In Figure
\ref{vf3d} node $N_i$ is the barycenter of the face defined by
vertices $V_1$, $V_2$ and $V_3$ (see Figure \ref{vf3d}). This is why we will call this
finite volume of {\it face-type}.

\begin{figure}[H]
	\begin{minipage}{0.5\linewidth}
		\begin{figure}[H]
			\begin{center}
				\vspace{1cm}
				\begin{picture}(100,90)
				\put(0,0){\makebox(100,90){
						\vspace{-2cm}
						\includegraphics[width=7.5cm]{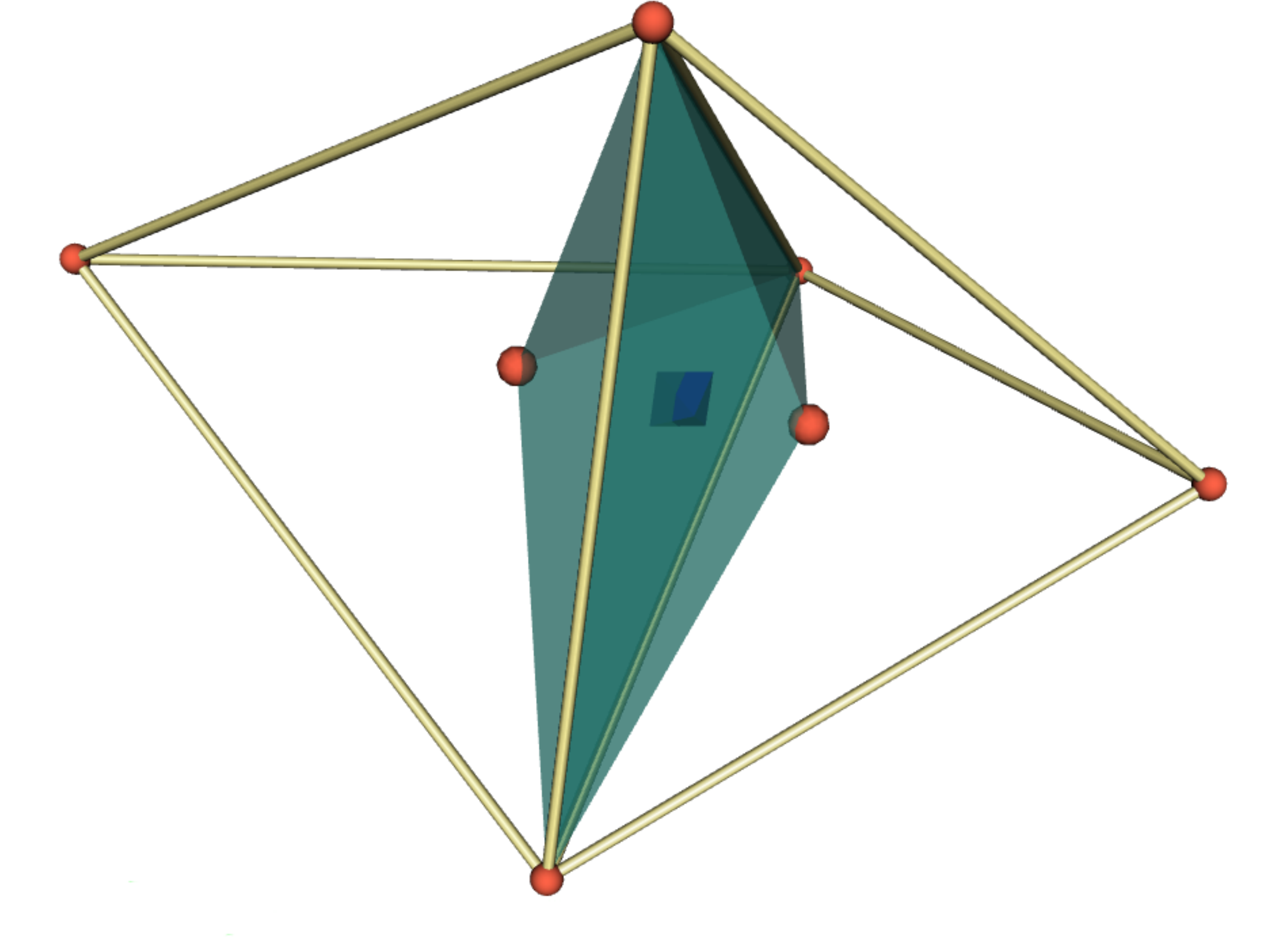}}}
				\put(35,-70){\normalsize{$V_{2}$}}
				\put(81,55){\normalsize{$V_{3}$}}
				\put(60,100){\normalsize{$V_{1}$}}
				\put(-50,62){\normalsize{$V_{4}$}}
				\put(160,12){\normalsize{$V_{4}'$}}
				\put(18,35){\normalsize{$B$}}
				\put(90,22){\normalsize{$B'$}}
				\put(58,37){\normalsize{$N_i$}}
				\end{picture}
				\vspace{2.5cm}
			\end{center}
		\end{figure}
	\end{minipage}\hfill
	\begin{minipage}{0.3\linewidth}
		\begin{figure}[H]
			\begin{center}
				\vspace{1cm}
				\begin{picture}(100,90)
				\put(0,0){\makebox(100,90){
						\vspace{-2cm}
						\includegraphics[width=4.5cm]{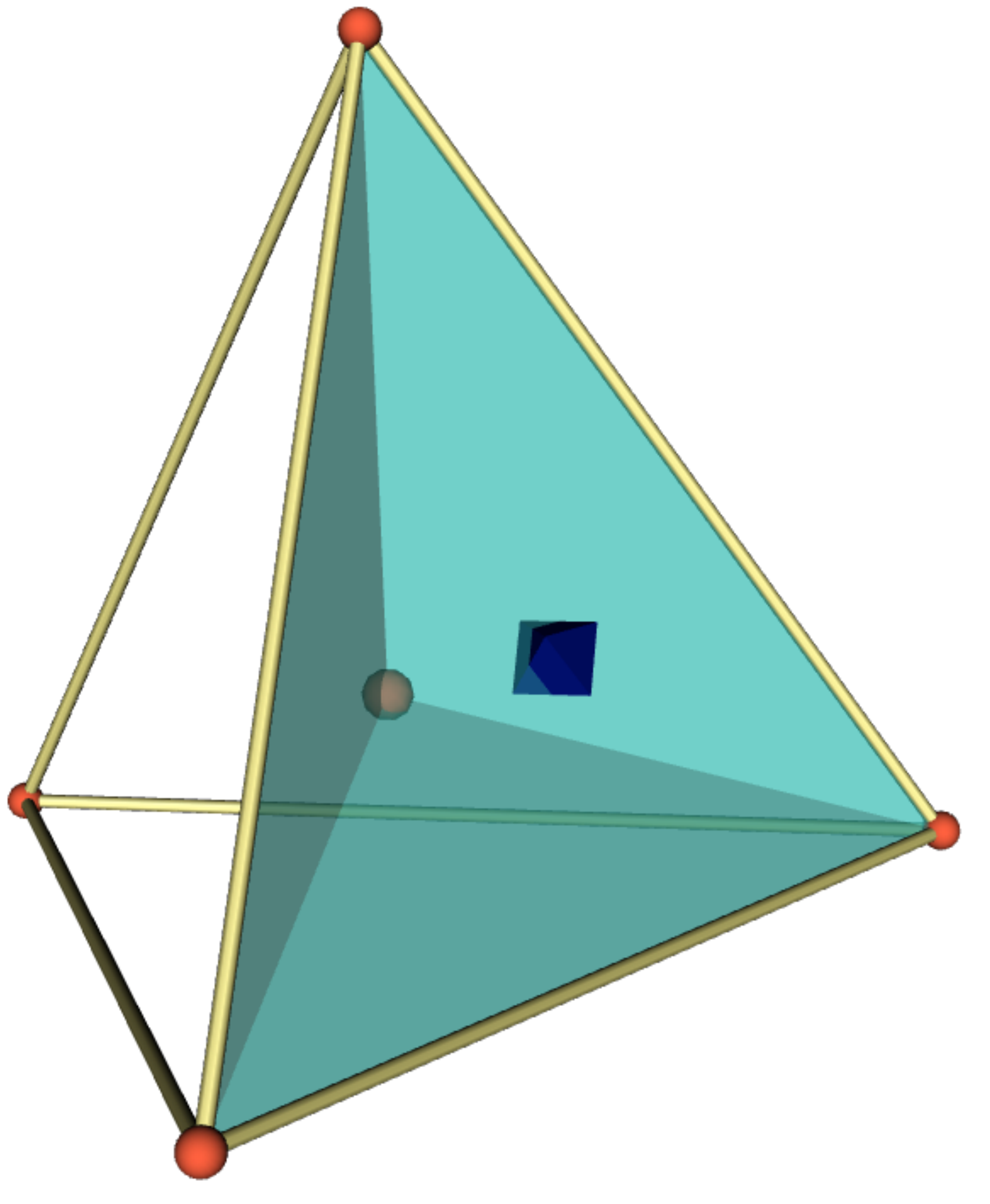}}}
				\put(12,-77){\normalsize{$V_{2}$}}
				\put(110,-33){\normalsize{$V_{3}$}}
				\put(29,100){\normalsize{$V_{1}$}}
				\put(-25,-12){\normalsize{$V_{4}$}}
				\put(25,0){\normalsize{$B$}}
				\put(58,15){\normalsize{$N_i$}}
				\end{picture}
				\vspace{2.5cm}
			\end{center}
		\end{figure}
	\end{minipage}
	\caption{Interior (left) and boundary (right) finite volumes of the face-type.}\label{vf3d}
\end{figure}

The notation employed is as follows:
\begin{itemize}
	\item Each interior node $N_i$ has as neighboring nodes the set ${\cal K}_{i}$
	consisting of the barycentres of the faces of the two tetrahedra of the initial mesh to which it belongs.
	\item The face $\Gamma_{ij}$ is the interface between cells $C_{i}$ and $C_{j}$. $N_{ij}$ is the barycentre of the face.
	\item The boundary of $C_{i}$ is denoted by
	$\Gamma_i =\displaystyle \bigcup_{N_j \in {\cal K}_{i}} \Gamma_{ij}$.
	\item Finally, $\widetilde{\boldsymbol{\eta}}_{ij} $ represents the outward unit normal vector to $\Gamma_{ij}$.
	We define $\boldsymbol{\eta_{ij}}:=\boldsymbol{\widetilde{\eta}_{ij}} ||\boldsymbol{\eta_{ij}} ||$, where,
	$||\boldsymbol{\eta_{ij}} ||:=\mbox{Area}(\Gamma_{ij})$.
\end{itemize}

\subsection{Finite volume discretization}
The discrete approximation of the conservative variables is taken to
be constant per finite volume, as it represents an integral average.
By integrating \eqref{eq:tidenincrec}, \eqref{eq:k_discret}, \eqref{eq:epsilon_discret}
and \eqref{eq:especies_discret}, on $C_{i}$ and applying the Gauss Theorem, we get
\begin{gather}{
\displaystyle \frac{1}{\Delta t}\left( \widetilde{\mathbf{W}}_{\mathbf{u},\, i}^{n+1} - \mathbf{W}_{\mathbf{u},\ i}^{n}\right)  + \frac{1}{\left| C_{i} \right|} \displaystyle \int_{\Gamma_i}\mathcal{F}^{\mathbf{w}_{\mathbf{u}}}\left({\mathbf{W}}^n_{\mathbf{u}}\right){\boldsymbol{\tilde{\eta}}_{i}}\mathrm{dS} 
 + \frac{1}{\left| C_{i} \right|} \displaystyle \int_{C_i}\nabla\pi^n \mathrm{dV}  -\frac{1}{\left| C_{i} \right|} \displaystyle\int_{\Gamma_i}\left(\tau^{n}\right)\boldsymbol{\tilde{\eta}}_{i} \mathrm{dS}  =\frac{1}{\left| C_{i} \right|} \int_{C_i} {\mathbf{f}^n_{\mathbf{u}}}\mathrm{dV}, \label{eq:wu_discret}}\\[0.4cm]{
\displaystyle \frac{1}{\Delta t}\left( \widetilde{W}_{k,\, i}^{n+1}-W_{k,\, i}^n\right) +\frac{1}{\left| C_{i} \right|}\displaystyle \int_{\Gamma_i}\mathcal{F}^{w_{k}}\left(W_{k}^n,\mathbf{U}^{n} \right)\boldsymbol{\tilde{\eta}}_{i}\mathrm{dS} 
\displaystyle -\frac{1}{\left| C_{i} \right|}\int_{\Gamma_i}\left[ \left( \mu+\frac{\mu_t^{n}}{\sigma_k} \right)\nabla \frac{W_k^n}{\rho} \right]\boldsymbol{\tilde{\eta}}_{i}\mathrm{dS}=0,}\\[0.4cm]{
\displaystyle \frac{1}{\Delta t}\left( \widetilde{W}_{\varepsilon,\, i}^{n+1}-W_{\varepsilon,\, i}^n\right) +\frac{1}{\left| C_{i} \right|}\displaystyle \int_{\Gamma_i}\mathcal{F}^{w_{\varepsilon}}\left(W_{\varepsilon}^n,\mathbf{U}^{n} \right)\boldsymbol{\tilde{\eta}}_{i}\mathrm{dS}
\displaystyle -\frac{1}{\left| C_{i} \right|}\int_{\Gamma_i}\left[ \left( \mu+\frac{\mu_t^{n}}{\sigma_{\varepsilon}} \right)\nabla \frac{W_{\varepsilon}^n}{\rho} \right]\boldsymbol{\tilde{\eta}}_{i}\mathrm{dS}=0,}\\[0.4cm]{
\displaystyle \frac{1}{\Delta t}\left( \widetilde{\mathbf{W}}_{\mathbf{y},\, i}^{n+1} - \mathbf{W}_{\mathbf{y},\, i}^n\right)  + \frac{1}{\left| C_{i} \right|} \displaystyle \int_{\Gamma_i} \mathcal{F}^{\mathbf{w}_{\mathbf{y}}}\left(\mathbf{W}_{\mathbf{y}}^n,\mathbf{U}^{n} \right)\boldsymbol{\tilde{\eta}}_{i}\mathrm{dS}  
- \frac{1}{\left| C_{i} \right|} \displaystyle\int_{\Gamma_i}  \left[ \left( \rho\mathcal{D}+\displaystyle \frac{\mu_t^{n}}{Sc_t} \right)\nabla  \left( \frac{1}{\rho}\mathbf{W}_{\mathbf{y}}^n\right) \right]\boldsymbol{\tilde{\eta}}_{i}\mathrm{dS} = 0,\label{eq:y_discret}}
\end{gather}
where $\left| C_{i} \right|$ denotes the volume of $C_{i}$ { and
$\boldsymbol{ \widetilde{\eta}}_{i}$ is the outward unit normal of $\Gamma_{i}$ at each point}.
Within the following sections we will detail how to approximate the former integrals.

\subsection{Numerical flux}
We define the global normal flux  on ${\Gamma}_{i}$ as
${\cal{Z} }(\mathbf{W}^{n},{  {\boldsymbol{\tilde{\eta}}}_{i}} ): ={
\mathbf{\cal{F} }} (\mathbf{W}^{n}) { {\boldsymbol{\tilde{\eta}}}_{i}}$.
Thanks to the shape  of the convective terms in equations \eqref{eq:wu_discret}-\eqref{eq:y_discret}
their integrals can be computed globally. We first split
$\Gamma_{i}$ into the cell interfaces $\Gamma_{ij}$, namely
\begin{equation}
\displaystyle  \int_{ \Gamma_{i} } \mathbf{\cal{F}}(
\mathbf{W}^{n}) { \boldsymbol{ \widetilde{\eta}}_{i}} \,\mathrm{dS}
=\displaystyle \sum_{N_j \in {\cal K}_{i}} \displaystyle
\int_{\Gamma_{ij}} {
	\cal{Z} }(\mathbf{W}^{n},{ {\boldsymbol{\tilde{\eta}}}_{ij}} )  \,\mathrm{dS}.
\end{equation}
Then, in order to obtain a stable discretization, the integral on
$\Gamma_{ij}$ is approximated by an upwind scheme using a numerical
flux function $\phi$:
\begin{equation}
\displaystyle \int_{\Gamma_{ij}}{
	\cal{Z} }(\mathbf{W}^{n},{ {\boldsymbol{\tilde{\eta}}}_{ij}} ) \mathrm{dS}
\approx  \phi \left(
\mathbf{W}_{i}^{n},\mathbf{W}_{j}^{n},\boldsymbol{
	\eta}_{ij}                   \right).
\end{equation}
The expression of $\phi$ depends on the upwind scheme. In this paper,
we consider the Rusanov scheme (see \cite{Rus62}){:}
\begin{gather}
\phi \left(
\mathbf{W}_{i}^{n},\mathbf{W}_{j}^{n},\boldsymbol{
	\eta}_{ij}                   \right)
= \displaystyle\frac{1}{2} ({\cal Z}
(\mathbf{W}_i^n,\boldsymbol{\eta}_{ij})+{\cal Z}(\mathbf{W}_j^n,\boldsymbol{\eta}_{ij}))\nonumber\\
-\frac{1}{2}
\alpha_{RS}(\mathbf{W}_i^n,\mathbf{W}_j^n,\boldsymbol{\eta}_{ij}) \left( \mathbf{W}_j^n-\mathbf{W}_i^n\right) ,
\end{gather}
where the coefficient $\alpha_{RS}$ can be computed in a coupled way
for all the equations, so that, it is defined by
\begin{equation}
\alpha_{RS}(\mathbf{W}_i^n,\mathbf{W}_j^n,\boldsymbol{\eta}_{ij}):=
\max \left\{2\left| \mathbf{U}_i^n\cdot {\boldsymbol{\eta}_{ij}} \right|,  2\left|
\mathbf{U}_j^n \cdot {\boldsymbol{\eta}_{ij}}\right| \right\}\label{eq:alphars_coupled}
\end{equation}
or, we can consider 
\begin{equation}\alpha^{\mathbf{w} _{\mathbf{u} }}_{RS}(\mathbf{W}_i^n,\mathbf{W}_j^n,\boldsymbol{\eta}_{ij})
:=\max \left\lbrace 2\left|\mathbf{U}_{i} \cdot{\boldsymbol{\eta}_{ij}}\right|,2\left|\mathbf{U}_{j}
\cdot{\boldsymbol{\eta}_{ij}}\right|\right\rbrace \label{eq:alphars_decoupled1}\end{equation}
for the momentum equation and
\begin{equation}
\hat{\alpha}_{RS}(\mathbf{W}_i^n,\mathbf{W}_j^n,\boldsymbol{\eta}_{ij}):=
\max \left\lbrace \left|\mathbf{U}_{i} \cdot{\boldsymbol{\eta}_{ij}}\right|,\left|\mathbf{U}_{j}
\cdot{\boldsymbol{\eta}_{ij}}\right|\right\rbrace\label{eq:alphars_decoupled2}\end{equation}
for the remaining equations.

Directly using the value obtained for the conservative variables at each node at the previous time step,
Rusanov scheme is first order in space and time. Two different methodologies will be introduced in order
to obtain second order schemes: the Kolgan-type scheme and the LADER scheme.

\subsubsection{Kolgan-type scheme}
Kolgan, \cite{Kol11}, introduced for the first time a non linear high order scheme that circumvents
Godunov's theorem.  In order to do that, he proposed the use of limited slopes in the reconstruction
of the conservative values used to build the flux function.
Following this work the CVC Kolgan-type scheme was introduced in \cite{CVC10} and \cite{CV12} for
the shallow water equations. The new scheme, which is second order accuracy in space and first order
in time, can be extended to the resolution of the Navier-stokes equations. This method is based on
the idea of replacing the conservative values $\mathbf{W}_{i}^{n}$, $\mathbf{W}_{j}^{n}$ in the
numerical viscosity by improved interpolations given by $\mathbf{W}_{iL}^n$, $\mathbf{W}_{jR}^n$
at both sides of each face $\Gamma_{ij}$,
\begin{equation*}
\mathbf{W}_{iL}^n=\mathbf{W}_{i}^n+\Delta^{ijL},\quad \mathbf{W}_{jR}^n=\mathbf{W}_{j}^n-\Delta^{ijR},
\end{equation*}
where $\Delta^{ijL}$ and $\Delta^{ijR}$ are the left and right limited slopes at the face defined
through the Galerkin gradients computed on the upwind tetrahedra, $T_{ijL} $ and $T_{ijR} $,
respectively (see Figure \ref{fig:up_g_t} for the 2D case). Moreover, we avoid spurious oscillations
by taking into account a minmod-type limiter (see \cite{Kol11}):
\begin{eqnarray}
&\;& \Delta^{ijL}=Lim\left( 
\frac{1}{2} \left( 
\nabla {\mathbf{W}^n }
\right)_{ T_{ijL} }  \overline{ N_{i} N_{j}}\,,\mathbf{W}_{j}^n-\mathbf{W}_{i}^n 
\right) ,\\
&\;& \Delta^{ijR}=Lim\left( 
\frac{1}{2} \left( \nabla {\mathbf{W}^n } \right)_{T_{ijR}}  \overline{ N_{i} N_{j}}\,,\mathbf{W}_{j}^n-\mathbf{W}_{i}^n
\right).
\end{eqnarray}
{ Following \cite{CV12},}
this high-order extrapolation is used only in the upwind contribution of the numerical flux
retaining the conservative variables in the centred part. So that, the numerical flux reads
\begin{gather}
\phi \left( \mathbf{W}_i^n,\mathbf{W}_j^n,\mathbf{W}_{iL}^n,\mathbf{W}_{jR}^n,\boldsymbol{\eta}_{ij} \right) \nonumber\\
= \displaystyle \frac{1}{2} \left( {\cal Z}(\mathbf{W}_i^n,\boldsymbol{\eta}_{ij})+{\cal Z}(\mathbf{W}_j^n,\boldsymbol{\eta}_{ij})\right)  
-\displaystyle\frac{1}{2} \alpha_{RS} \left( (\mathbf{W}_{iL}^n,\mathbf{W}_{jR}^n,\boldsymbol{\eta}_{ij} \right) \left( \mathbf{W}_{jR}^n-\mathbf{W}_{iL}^n\right).
\end{gather}
\begin{figure}
	\centering
	\includegraphics[width=0.33\linewidth]{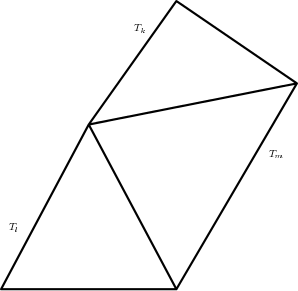}\hspace{-0.13cm}
	\includegraphics[width=0.33\linewidth]{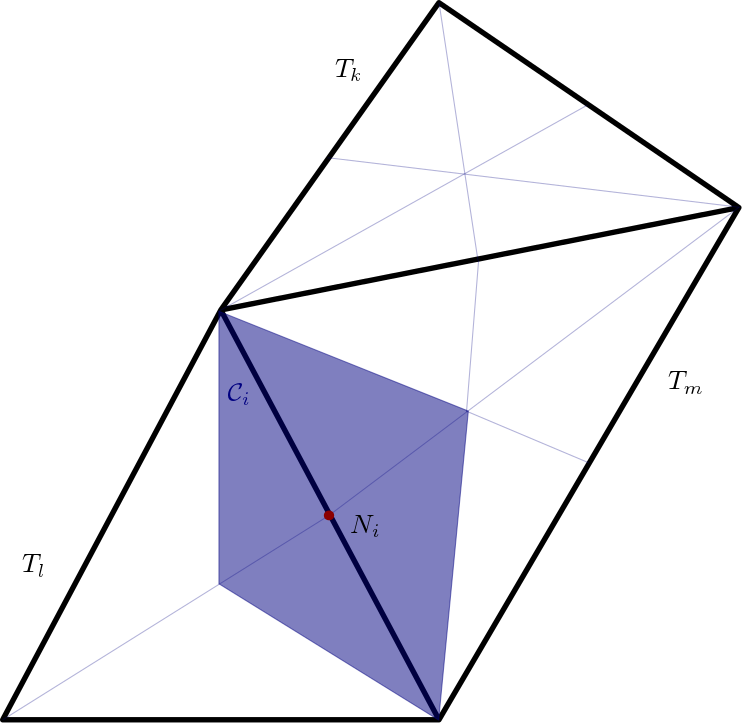}\hspace{-0.13cm}
	\includegraphics[width=0.33\linewidth]{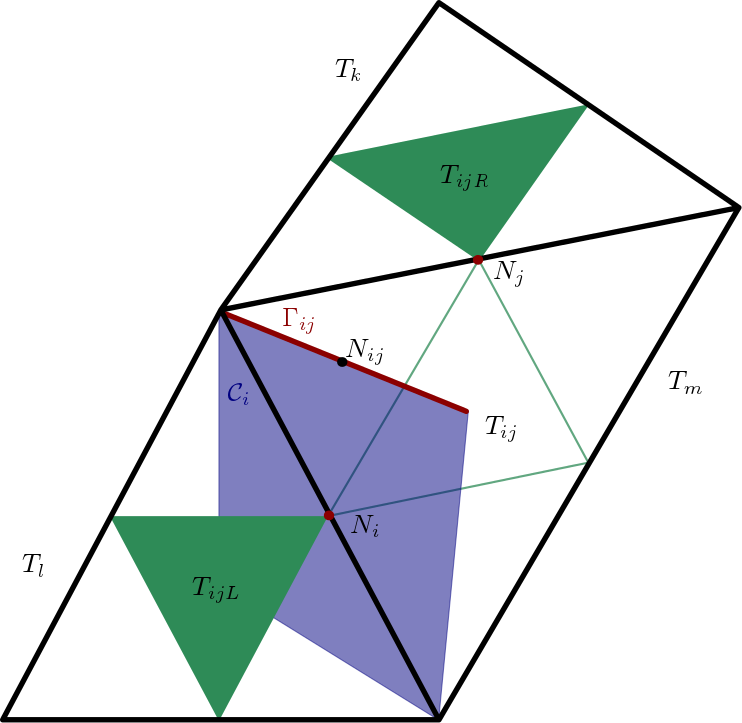}
	{\caption{Construction of a dual 2D mesh and auxiliary triangles.
			Left: Finite elements of the original triangular mesh (black).
			Centre: Finite volume $C_{i}$ (purple).
			Right: Upwind triangles (green).}\label{fig:up_g_t}}
\end{figure}
\subsubsection{LADER}\label{sec:LocalADER}
ADER methodology, first put forward in \cite{TMN01} for the linear advection equation on
Cartesian meshes and extended, for the one-dimensional case, in \cite{BTVC16} to account
for the diffusion and reaction terms, is applied in order to obtain a second order in time
and space scheme. More precisely, a modification of the method is proposed to profit from
the dual mesh structure to reduce the size of the stencil and, hence, the computational cost. 
Due to the small neighbourhood involved on the calculations related to each node, we will name
this new scheme as LADER. The LADER scheme for the one-dimensional advection-diffusion-reaction
equation is presented in  \ref{sec:appendix}; the stability and truncation error analysis are also given.

To extend LADER to the three dimensional case four relevant issues must be taken into account:
\begin{enumerate}
	\item\label{enu:uno} The advection depends on the diffusion terms. That is, within { the}
	computation of the flux we will use Cauchy-Kovalevskaya procedure to obtain evolved values for the
	variables which contain information from both terms. \item\label{enu:dos} The evolved variables
	obtained for computing the diffusion term neglect the presence of the advection term.
	\item As a consecuence of \ref{enu:uno} and \ref{enu:dos}, advection and diffusion terms need to be
	computed using the proper evolved variables, which will be different for each of them.
	\item To compute the gradients needed to obtain the evolved variables we can profit from the
	FE mesh and use a Galerkin approach.
\end{enumerate}

{ In this section, for ease of comprehension,} we will assume that the diffusion term
only accounts for the velocity gradient, the remaining terms can be computed analogously.
Let us consider $W$ an approximation of a scalar conservative variable and $\alpha$ the related
diffusion coefficient, the proposed method includes the following steps:
\begin{description}
	\item[ Step 1.] Data reconstruction.
	Reconstruction of the data in terms of first degree polynomials is considered. 
	At each finite volume we define four polynomials each of them at the neighbourhood
	of one of the boundary faces. 
	Focusing on a face $\Gamma_{ij}$ its two related reconstruction polynomials are
	\begin{equation}p^{i}_{ij}(N)=W_i+(N-N_{i})\left( \nabla W\right)^{i}_{ij},\quad p^{j}_{ij}(N)=W_j+(N-N_{j})\left( \nabla W\right)^{j}_{ij}.\end{equation}
	A possible election of the gradients is
	\begin{equation}\left( \nabla W\right)^{i}_{ij}= \nabla  W_{ T_{ijL} },\quad \left( \nabla W\right)^{j}_{ij}= \nabla  W_{ T_{ijR} }.\end{equation}
	which will result on a linear reconstruction as it is based on a fixed stencil. 
	
	In order to circumvent Godunov's theorem and prevent spurious oscillations,
	a non-linear reconstruction is considered (see \cite{Toro}). More precisely,
	the ENO (Essentially Non-Oscillatory) interpolation method is applied.
	The slopes are adaptively chosen as follows:
	\begin{gather*}
	\left( \nabla W\right)^{i}_{ij}= \left\lbrace
	\begin{array}{lr}
	\left(\nabla W \right)_{T_{ijL}} & \textrm{if }\left| \left(\nabla W \right)_{ T_{ijL}}\cdot \left(N_{ij}-N_{i}\right)\right| \leq \left|\left(\nabla W \right)_{ T_{ij}}\cdot \left(N_{ij}-N_{i}\right) \right|,\\[8pt]
	\left(\nabla W \right)_{ T_{ij}} & \textrm{if }\left| \left(\nabla W \right)_{ T_{ijL}}\cdot \left(N_{ij}-N_{i}\right)\right| > \left| \left(\nabla W \right)_{ T_{ij}}\cdot \left(N_{ij}-N_{i}\right) \right|;
	\end{array}
	\right.
	\end{gather*}
	\begin{gather*}
	\left( \nabla W \right)^{j}_{ij}= \left\lbrace
	\begin{array}{lr}
	\left(\nabla W\right)_{ T_{ijR}} & \textrm{if }\left|\left( \nabla W \right)_{ T_{ijR}}\cdot \left(N_{ij}-N_{j}\right)\right| \leq \left|\left(\nabla W \right)_{ T_{ij}}\cdot \left(N_{ij}-N_{j}\right)\right|,\\[8pt]
	\left(\nabla W\right)_{ T_{ij}}  & \textrm{if }\left| \left(\nabla W \right)_{ T_{ijR}}\cdot \left(N_{ij}-N_{j}\right)\right| > \left|\left( \nabla W\right) _{ T_{ij}}\cdot \left(N_{ij}-N_{j}\right) \right|;
	\end{array}
	\right.
	\end{gather*}
	where $\left( \nabla W\right)_{ T_{ij}} $ is the gradient of the velocity
	at the auxiliary tetrahedra which intersects the face.	
	\item[Step 2.] Computation of boundary extrapolated values at the barycenter of the faces, $N_{ij}$:
	\begin{gather} W_{i\, N_{ij}} =p^{i}_{ij}(N_{ij})
	= W_{i}+(N_{ij}-N_{i})   \left( \nabla W\right)^{i}_{ij},\\
	W_{j\, N_{ij}}=p^{j}_{ij}(N_{ij})
	W_{j}+(N_{ij}-N_{j})   \left( \nabla W\right)^{j}_{ij}.  \end{gather}
	
	\item[Step 3.] 
	Computation of the flux terms with second order of accuracy using the mid-point rule.
	Taylor series expansion in time$\,$ and$\,$ Cauchy$\,$-$\,$Kovalevskaya  procedure
	are applied to locally approximate the conservative variables at time
	$\frac{\Delta t}{2}$. This methodology accounts for the contribution of the advection
	and diffusion terms to the time evolution of the flux term.
	The resulting evolved variables read
	\begin{gather}\overline{W_{i\, N_{ij}}}=W_{i\, N_{ij}}-\frac{\Delta t}{2\mathcal{L}_{ij}}\left( \mathcal{Z}(W_{i\, N_{ij}},\boldsymbol{\eta}_{ij}  )+\mathcal{Z}(W_{j\, N_{ij}},\boldsymbol{\eta}_{ij}  )\right) \notag\\ +
	\frac{\Delta t}{2\mathcal{L}_{ij}^2} \left(   \alpha_{iN_{ij}} \left( \nabla W\right)^{i}_{ij}\boldsymbol{\eta}_{ij}  + \alpha_{j N_{ij}}  \left( \nabla W\right)^{j}_{ij}\boldsymbol{\eta}_{ij} \right), \\
	\overline{ W_{j\, N_{ij}}}= W_{j\, N_{ij}}-\frac{\Delta t}{2\mathcal{L}_{ij}}\left( \mathcal{Z}( W_{i\, N_{ij}},\boldsymbol{\eta}_{ij}  )+\mathcal{Z}( W_{j\, N_{ij}},\boldsymbol{\eta}_{ij}  )\right)\notag \\+
	\frac{\Delta t}{2\mathcal{L}_{ij}^2} \left(   \alpha_{iN_{ij}}  \left( \nabla W\right)^{i}_{ij}\boldsymbol{\eta}_{ij}  + \alpha_{j N_{ij}}  \left( \nabla W\right)^{j}_{ij} \boldsymbol{\eta}_{ij} \right) .\end{gather}
	We have denoted $\mathcal{L}_{ij}=\min\left\lbrace\frac{{\left| C_{i} \right|}}{\mathrm{S}(C_i)},\frac{{\left| C_{j} \right|}}{\mathrm{S}(C_j)}\right\rbrace$ with $\mathrm{S}(C_i)$ the area of the surface of cell $C_i$.
	Two different options will be considered in the scheme concerning the evolved variables.
	The first one corresponds to the previous definition of the evolved variables.
	Meanwhile, the second one neglects the evolution of the diffusion term.
	
	\item[Step 4.] Computation of the numerical flux considering Rusanov scheme.
	\begin{gather}
	\phi\left(\overline{ W^n_{i\, N_{ij}}},\overline{ W^n_{j\, N_{ij}}},\boldsymbol{\eta}_{ij}\right)=
	\frac{1}{2}\left(\mathcal{Z}\left(\overline{ W^n_{i\, N_{ij}}},\boldsymbol{\eta}_{ij}\right)+\mathcal{Z}\left(\overline{ W^n_{j\, N_{ij}}},\boldsymbol{\eta}_{ij}\right)\right) \nonumber\\
	-\frac{1}{2}\alpha_{RS}\left(\overline{ W^n_{i\, N_{ij}}},\overline{ W^n_{j\, N_{ij}}},\boldsymbol{\eta}_{ij}\right)\left(\overline{ W^n_{j\, N_{ij}}}-\overline{ W^n_{i\, N_{ij}}}\right).
	\end{gather}
\end{description}

\subsection{Viscous terms}
We come next to describe the computation of the integrals involving the viscous terms.
First, applying Gauss' theorem we relate the volume integral of the diffusion term with
a surface integral over the boundary, $\Gamma_{i}$. Next, this integral is split into the integrals over the cell interfaces
$\Gamma_{ij}$. Thus, the viscous term of the momentum conservation equation reads 
\begin{gather}
\int_{C_i}\mathrm{div }\tau^n\mathrm{dV} = \sum_{N_j\in\mathcal{K}_i}\int_{\Gamma_{ij}} \tau^n\tilde{\boldsymbol{\eta}}_{ij} \, \mathrm{dS} \nonumber\\
= \sum_{N_j\in\mathcal{K}_i}\int_{\Gamma_{ij}}\left[
\left( \mu+ \mu_t^n\right) \left(\nabla \mathbf{U}^{n}+\left( \nabla\mathbf{U}^{n}\right)^T\right)   - \frac{2}{3}\rho K^n I\right]{\boldsymbol{\widetilde{\eta}}_{ij}} \,\mathrm{dS}.\label{eq:visc_term_int}
\end{gather}
Two different approaches can be considered in order to compute the above integral.

On the one hand, decomposition with semi-implicit and explicit discretizations can
be applied when using the CVC Kolgan-type scheme. This methodology splits the diffusion
flux into its orthogonal and non-orthogonal parts and relax the stability condition
on the time step size (see \cite{BFSV14} for further details).

On the other hand, the dual mesh ease the use of Galerkin approach to compute
the  derivatives involved in \eqref{eq:visc_term_int}. 
We introduce a numerical diffusion function $\varphi_{\mathbf{u}}$ such that
\begin{gather}\int_{\Gamma_{ij}}
\left( \mu+ \mu_t^n\right) \nabla \mathbf{U}^n\tilde{\boldsymbol{\eta}}_{ij}\mathrm{dS}
\approx \varphi_{\mathbf{u}}\left( \mathbf{U}_{i}^n,\mathbf{U}_{j}^n,\mu_{t,\, i}^n,\mu_{t,\, j}^n, \boldsymbol{\eta}_{ij} \right) \end{gather}
and we consider
\begin{gather}
\varphi_{\mathbf{u}}\left( \mathbf{U}_{i}^n,\mathbf{U}_{j}^n,\mu_{t,\, i}^n,\mu_{t,\, j}^n, \boldsymbol{\eta}_{ij} \right)
=\left( \mu+ \mu_{t,\, ij}^n\right) \left( \nabla \mathbf{U}^n\right)_{T_{ij}}\boldsymbol{\eta}_{ij}, \label{eq_psi_viscterm_grad} \end{gather}
with
\begin{gather}
\mu_{t,\,ij}^n=\frac{1}{2}\left(\mu_{t,\,i}^n+\mu_{t,\,j}^n  \right). \end{gather}
{ Since we know the value of the turbulent kinetic energy at the nodes
	of the finite volumes, we approximate the turbulent kinetic energy term as
	the average of the values obtained at the two nodes related to the face}
\begin{equation}
-\int_{\Gamma_{ij}}\frac{2}{3} W^n_k \tilde{\boldsymbol{\eta}}_{ij}\mathrm{dS} =
-\frac{1}{3} \left(W^n_{k, \,i}+W^n_{k,\, j}\right)\boldsymbol{\eta}_{ij}.
\end{equation}

Finally, the viscous terms for the remaining equations are obtained equally to the gradient term of the momentum equation:
\begin{eqnarray}
\int_{C_i}\frac{1}{\rho} \mathrm{div }\left( \boldsymbol{\mathcal{D}}^{n} \nabla \widehat{\mathbf{W}}^{n} \right)\mathrm{dV} = \frac{1}{\rho}\sum_{N_j\in\mathcal{K}_i}\int_{\Gamma_{ij}} \boldsymbol{\mathcal{D}}^{n} \nabla \widehat{\mathbf{W}}^{n}  \, \tilde{\boldsymbol{\eta}}_{ij} \, \mathrm{dS}, 
\end{eqnarray}
where
\begin{eqnarray}
\boldsymbol{\mathcal{D}}^{n} =  \left(  \begin{array}{cccc}\mathcal{D}_{k}^{n} & 0 & 0  \\
0 & \mathcal{D}_{\varepsilon}^{n} & 0 \\
0 & 0 & \mathcal{D}_{{ \mathbf{y}}}^{n}
\end{array} \right) 
=  \left(  \begin{array}{cccc}
\mu + \displaystyle \frac{\mu_t^{n}}{\sigma_k}  & 0 & 0 \\
0 & \mu + \displaystyle \frac{\mu_t^{n}}{\sigma_\varepsilon}  & 0  \\
0 & 0 & \rho\mathcal{D} + \displaystyle \frac{\mu_t^{n}}{Sc_t}
\end{array}  \! \right)\! \! .
\end{eqnarray} 
Thus, we can introduce the diffusion flux function, $\varphi_{\widehat{\mathbf{w}}}$, verifying
\begin{gather}
\int_{\Gamma_{ij}} \boldsymbol{\mathcal{D}}^{n} \nabla \widehat{\mathbf{W}}^{n} { \tilde{\boldsymbol{\eta}}_{ij}}  \mathrm{dS} \approx\varphi_{\widehat{\mathbf{w}}}\left( 
\widehat{\mathbf{W}}_{i}^n,\widehat{\mathbf{W}}_{j}^n,
\mu_{t,\, i}^n,\mu_{t,\, j}^n, \boldsymbol{\eta}_{ij}
\right), \notag\\
\varphi_{\widehat{\mathbf{w}}}\left( 
\widehat{\mathbf{W}}_{i}^n,\widehat{\mathbf{W}}_{j}^n,
\mu_{t,\, i}^n,\mu_{t,\, j}^n, \boldsymbol{\eta}_{ij}
\right)= \boldsymbol{\mathcal{D}}^{n}_{ij} \left( \nabla \widehat{\mathbf{W}^{n}} \right)_{T_{ij}} \boldsymbol{\eta}_{ij}\label{eq:dif_funct_what}
\end{gather}

The above methodologies are used to approximate the viscous terms when choosing
a first order method or the CVC Kolgan-type scheme to compute the advection term.
Nevertheless, the LADER methodology requires a special treatment.

\subsubsection{LADER approach: the viscous terms }
As was already introduced in Section \ref{sec:LocalADER}, to apply LADER and to
obtain a second order in space and time scheme, instead of computing the diffusion
flux functions, $\varphi_{\mathbf{u}}$ and $\varphi_{\widehat{\mathbf{W}}}$, with the
value of the variables at the previous time step, $\mathbf{U}^n$, $K^n$, $E^n$ and $\mathbf{Y}^n$,
its is necessary to use some evolved values, $\overline{\mathbf{U}^n}$, $\overline{K^n}$, $\overline{E^n}$ and $\overline{\mathbf{Y}^n}$. 

It is important to remark that the former evolved variables do not match the already
computed ones for the flux term { (see \ref{sec:appendix} for a detailed
	analysis of the scalar advection-diffusion-reaction equation)}.
Taylor series expansion in time and Cauchy-Kovalevskaya procedure are applied neglecting
the advection term { so that a second order in space and time scheme is attained}:
\begin{gather}
{ \overline{\mathbf{U}^n}=\mathbf{U}^n+ \frac{\Delta t}{2}\left\lbrace  \mathrm{div} \left[\left(\mu+\mu_t^n\right) \nabla\mathbf{U}^n -\frac{2}{3}\rho K^n\mathbf{I}\right]   \right\rbrace,}\label{eq:evol_u_dif}\\
\overline{K^n}=K^n+ \frac{\Delta t}{2}\left[ \left(\mu+ \frac{\mu_t^n}{\sigma_{k}}\right)\nabla K^{n} \right],\label{eq:evol_k_dif}\\
\overline{E^n}=E^n+ \frac{\Delta t}{2}\left[ \left(\mu+ \frac{\mu_t^n}{\sigma_{\varepsilon}}\right) \nabla E^{n}\right],\label{eq:evol_eps_dif}\\
\overline{\mathbf{Y}^n}=\mathbf{Y}^n+ \frac{\Delta t}{2}\left[ \left(\rho\mathcal{D}+ \frac{\mu_t^n}{Sc_t}\right) \nabla \mathbf{Y}^{n}\right].\label{eq:evol_y_dif}
\end{gather}

In what follows, we describe the computation of the evolved velocities at an arbitrary node $N_{i}$:
\begin{enumerate}
	\item The gradients of the original variables are computed at each auxiliary tetrahedra
	of the FE mesh, $T_{ij}$ (see, on the 2D representation in Figure \ref{fig:up_g_t}, 
	the triangle  with green contour). 	The value of the gradient at each node, $N_i$,
	is obtained as the average of the values on the two tetrahedra containing the node,
	$T_{ijL}$ (green filled triangle in Figure \ref{fig:up_g_t}) and $T_{ij}$.
	Taking into account the viscosity coefficients and the turbulent kinetic energy term,
	we introduce the auxiliary variable:
	\begin{equation}
	\mathbf{V}^{n}_{i}:=\left(\mu + \mu_{t,\,i}^{n}\right) \frac{1}{2} \left(\left( \nabla \mathbf{U}^{n}\right)_{ T_{ijL}} +\left( \nabla \mathbf{U}^{n}\right)_{ T_{ij}}\right)- \frac{2}{3}\rho K^{n}_{i} I.
	\end{equation}
	\item The divergence is computed as the average of the divergences of $\mathbf{V}^{n}$
	obtained on the auxiliary tetrahedra:
	\begin{equation}
	\overline{\mathbf{U}_i^{n}}=\mathbf{U}_i^{n} + \frac{\Delta t}{4}\mathrm{tr}\left(\left( \nabla \mathbf{V}^{n}\right)_{T_{ijL}}+\left( \nabla \mathbf{V}^{n}\right)_{T_{ij}}\right).
	\end{equation}
	\item The diffusion function $\varphi_{\mathbf{u}}$ is evaluated on the evolved variables:
	\begin{equation}
	\varphi_{\mathbf{u}}\left( \overline{\mathbf{U}_i^{n}},\overline{\mathbf{U}_j^{n}},\mu_{t,\, i}^n,\mu_{t,\, j}^n, \boldsymbol{\eta}_{ij} \right)
	=\left( \mu+ \mu_{t,\, {ij}}^n\right)\left( \nabla \overline{\mathbf{U}^n}\right)_{T_{ij}}\boldsymbol{\eta}_{ij}. \end{equation}
\end{enumerate}
The remaining evolved variables are similarly obtained. Hence the related diffusion function reads  
\begin{gather}
\varphi_{\widehat{\mathbf{w}}}\left( 
\overline{\widehat{\mathbf{W}}_{i}^n},\overline{\widehat{\mathbf{W}}_{j}^n},
\mu_{t,\, i}^n,\mu_{t,\, j}^n, \boldsymbol{\eta}_{ij}
\right)= \boldsymbol{\mathcal{D}}^{n}_{ij} \left( \nabla \overline{\widehat{\mathbf{W}}^{n}} \right)_{T_{ij}} \boldsymbol{\eta}_{ij}.
\end{gather}

\subsection{Pressure term}
For the integral of the pressure gradient we follow \cite{BFSV14}.
We split the boundary $\Gamma_{i}$ into the cell interfaces
$\Gamma_{ij}$ using { Gauss' theorem} and we compute the
pressure as the arithmetic mean of its values at the three vertices
of face ${\Gamma}_{ij}$ and the barycentre of the tetrahedra to which
the face belongs. Then, the corresponding approximation
of the integral is given by
\begin{equation}
\int_{ {\Gamma}_{ij} } \pi^n \, \boldsymbol{\widetilde{\eta}}_{ij} \mathrm{dS} \approx
\left[\frac{5}{12} \left( \pi^n(V_1)+\pi^n(V_2)\right)+ \frac{1}{12} \left(\pi^n(V_3)+
\pi^n(V_4)  \right)\right]  \boldsymbol{\eta}_{ij}.
\end{equation}

\subsection{Projection stage}
Within the projection stage, the pressure is computed using a standard
finite element method. The incremental projection method presented in
\cite{Guer06} is adapted to solve \eqref{eq:tideincre2}-\eqref{eq:qincrec}
obtaining the following weak problem:

{\textit{Find $\delta^{n+1}\in V_0:=\left\{ z\in H^1(\Omega): \int_{\Omega}z=0\right\}$ verifying }}
\begin{equation}
\displaystyle\int_{\Omega} \nabla \delta^{n+1}\cdot \nabla z  \, \mathrm{dV}= \displaystyle \frac{1}{\Delta t}
\int_{\Omega}\mathbf{\widetilde{W}}^{n+1} \cdot \nabla z  \, \mathrm{dV}  - \displaystyle \frac{1}{\Delta t} \int_{\partial \Omega} G^{n+1}z  \, \mathrm{dS}\quad \forall z\in V_0, 
\end{equation}
where $\delta^{n+1}:= \pi^{n+1}-\pi^{n}$ (see \cite{BFSV14} for further details).

\subsection{Post-projection stage}
Once the pressure is computed, we can update $\mathbf{W}_{\mathbf{u}}^{n+1}$
with $\nabla\delta^{n+1}_i$, that is,
\begin{equation}
\mathbf{W}^{n+1}_{\mathbf{u},\, i}=\widetilde{\mathbf{W}}^{n+1}_{\mathbf{u},\, i}+\Delta t \nabla \delta^{n+1}_i.
\end{equation}
The previous computation of the updated velocities allows for an implicit
approach of the production term $G_k$ on the turbulence equations.
Meanwhile, for the dissipative terms a semi-implicit scheme is used:
\begin{gather}
\frac{W_{k,\, i}^{n+1}-\widetilde{W}_{k,\, i}^{n+1}}{\Delta t}+W^n_{\varepsilon,\, i}-G_{k,\, i}(\mathbf{U}^{n+1})=f_{k,\, i}^n,\\
\displaystyle\frac{W_{\varepsilon,\, i}^{n+1}-\widetilde{W}_{\varepsilon,\, i}^{n+1}}{\Delta t}+C_{2\varepsilon}\frac{W_{\varepsilon,\, i}^n }{W_{k,\, i}^n}W_{\varepsilon,\, i }^{n+1}  \displaystyle- C_{1\varepsilon}\frac{W^n_{\varepsilon,\, i}}{W^n_{k,\, i}}G_{k,\, i}(\mathbf{U}^{n+1})=f_{\varepsilon,\, i}^n
\end{gather} 
where the derivatives involved in the production term, $G_{k,\, i}(\mathbf{U}^{n+1})$,
are computed as the averaged of the auxiliary tetrahedra related to the node $N_{i}$.
Finally, the source terms $\mathbf{f}_{\widehat{\mathbf{w}}}$ are pointwise evaluated.

\subsection{Boundary conditions}
The boundary conditions were defined following \cite{BFSV14}:
\begin{itemize}
	\item Dirichlet boundary conditions for inviscid fluids: the normal component of
	the conservative variable is set at the boundary nodes.
	\item Dirichlet boundary conditions for viscous fluids: the value of the
	conservative variable is imposed at the boundary nodes.
	\item Neumann boundary conditions: the definition of $\widetilde{\mathbf{W}}^{n+1}$
	takes into account the inflow/outflow boundary condition with no need for any additional treatment. 
\end{itemize}
Moreover, in the manufactured tests designed to analyse the order of accuracy
of the numerical discretizations, it is a usual practice to impose the values
of the exact solution at the boundary nodes. This practice avoids that the
accuracy of the method can be affected by the treatment of the boundary
conditions. From the mathematical point of view, it is like considering
Dirichlet boundary conditions.

\section{Numerical results} \label{sec:numer_res}
In this section, we present the results obtained for several test problems.
In order to define the time step, two different options are implemented in the code.
On the one hand, we can simply introduce a fixed time step. On the other hand,
we can provide the CFL from which the code will compute the time step at each
{ time} iteration. { The latest option is the one chosen to
	run the test cases presented in this paper. Therefore, to determine the time
	step at each time iteration, we compute a local value for the time step at each cell $C_{i}$,
\begin{equation}
	\Delta t_{C_{i}}=  \frac{\CFL\, \mathcal{L}_{i}^2}{2\left|\mathbf{U}_{i} \right| \mathcal{L}_{i}
		+ \max \left\lbrace \mu+\mu_{t,\, i}, \rho\mathcal{D}+\dfrac{\mu_{t,\, i}}{Sc_{t}} \right\rbrace }
\end{equation}
with $ \mathcal{L}_{i}:=\frac{{\left| C_{i} \right|} }{\mathrm{S} \left(C_{i}\right)}$.
Finally, as global time step at each time iteration, $\Delta t$,  we choose the minimum time steps obtained at each cell.

\begin{remark2}
The above definition of $\Delta t_{C_{i}}$ is valid if the transport of species equation is solved,
otherwise its value is given by
\begin{equation}
\Delta t_{C_{i}}=  \frac{\CFL \, \mathcal{L}_{i}^2}{2 \left|\mathbf{U}_{i} \right| \mathcal{L}_{i} + \mu+\mu_{t,\, i}  }.
\end{equation}
\end{remark2}
}

\subsection{Manufactured test 1. Laminar flow}\label{sec:pb_mms_laminar}
The first test to be posed was obtained using the method of the manufactured solutions (MMS).
We consider the domain $\Omega=\left[0,1\right]^3$ and we assume the flow being defined by
\begin{gather}
\rho =	1,\\[6pt]
\pi(x,y,z,t) =	\cos(\pi t (x+y+z)),\\[6pt]
\mathbf{u}(x,y,z,t) =\left(  \sin(\pi y t)\cos(\pi z t),\;-\cos(\pi z^3 t),\;	\exp(-2\pi x t^2)\right)^T,
\end{gather}
with $\mu=10^{-2}$. { The related source terms are included in \ref{sec:appendix_sourceterms}.}

To perform the error and order of accuracy analysis we employ the three uniform meshes with different
cell sizes presented in Table \ref{tab:mmstest_mesh}.
\begin{table}[h]\begin{center}
	\begin{tabular}{|c||c|c|c|c|c|c|}
		\hline Mesh  & $N$ & Elements & Vertices & Nodes & $V_h^m$ (m$^3$) & $V_h^M$ (m$^3$) \\\hline
		\hline $M_1$ & $4$ & $384$ & $ 125 $ & $ 864 $ & $ 6.51E-04 $ & $ 1.30E-03 $ \\
		\hline $M_2$ & $8$ & $ 3072 $ & $ 729 $ & $ 6528 $ & $ 8.14E-05 $ & $ 1.63E-04 $ \\
		\hline $M_3$ & $16$ & $ 24576 $ & $ 4913 $ & $ 50688 $ & $ 1.02E-05 $ & $ 2.03E-05 $ \\
		\hline
	\end{tabular}
	\caption{Manufactured test 1. Laminar flow.  Mesh features. }\label{tab:mmstest_mesh}
	\end{center}
\end{table}
We have denoted $N+1$ the number of points along the edges, $h=1/N$, $V_h^m$ the minimum volume
of the finite volumes and $V_h^M$ the maximum volume of the finite volumes. 

Four different methods are used to solve the problem: the first order method presented in \cite{BFSV14},
CVC method with an orthogonal decomposition of the diffusion term (CVC-orth), CVC method combined with
a Galerkin approach for the diffusion term (CVC-G) and LADER. The errors and orders, depicted in Table
\ref{mms_laminar_err_ord}, were computed as follows:
\begin{gather}
E(\pi)_{M_i} = \left\|\pi-\pi_{M_i} \right\|_{l^2(L^2(\Omega))}\quad E(\mathbf{w}_{\mathbf{u}})_{M_i}
= \left\|\mathbf{w}_{\mathbf{u}}-\mathbf{w}_{\mathbf{u}\, M_i} \right\|_{l^2(L^2(\Omega)^{3})},
\end{gather}
\begin{gather}
o_{\pi_{M_i/M_j}} = \frac{\log\left( E(\pi)_{M_i}/E(\pi)_{M_j}\right) }{\log\left( h_{M_i}/h_{M_j}\right) },\quad
o_{\mathbf{w}_{\mathbf{u}\, M_i/M_j}} = \frac{\log\left( E(\mathbf{w}_{\mathbf{u}})_{ M_i}/E(\mathbf{w}_\mathbf{u})_{M_j}\right) }{\log\left( h_{M_i}/h_{M_j}\right) }.\end{gather}
\begin{table}
	\begin{tabular}{|c|c||c|c|c||c|c|}
		\hline Method & Variable & $E_{M_1}$ & $E_{M_2}$ & $E_{M_3}$ & $o_{M_1/M_2}$ & $o_{M_2/M_3}$ \\\hline\hline
		\multirow{2}{*}{Order 1}      &$\pi$               & $1.24E-01$ & $5.70E-02$ & $2.96E-02$ & $1.12$ & $0.95$ \\
		&$\mathbf{w}_{\mathbf{u}}$         & $6.40E-02$ & $3.32E-02$ & $1.78E-02$ & $0.95$ & $0.90$ \\\hline
		\multirow{2}{*}{CVC-orth.}  &$\pi$               & $6.30E-02$ & $1.91E-02$ & $8.84E-03$ & $1.72$ & $1.11$ \\
		&$\mathbf{w_{u}}$         & $5.51E-02$ & $2.06E-02$ & $8.95E-03$ & $1.42$ & $1.20$ \\\hline
		\multirow{2}{*}{CVC-G}      &$\pi$               & $5.98E-02$ & $1.58E-02$ & $4.58E-03$ & $1.92$ & $1.78$ \\
		&$\mathbf{w_{u}}$         & $5.41E-02$ & $1.88E-02$ & $6.52E-03$ & $1.52$ & $1.53$ \\\hline
		\multirow{2}{*}{LADER}   &$\pi$               & $4.10E-02$ & $8.74E-03$ & $2.03E-03$ & $2.23$ & $2.11$ \\
		&$\mathbf{w_{u}}$         & $2.61E-02$ & $5.76E-03$ & $1.24E-03$ & $2.18$ & $2.22$ \\\hline
	\end{tabular}
	\caption{ Manufactured test1. Laminar flow.  Observed errors and convergence rates. $\CFL=1$.}\label{mms_laminar_err_ord}
\end{table}
We can observe that CVC-G method provides an order of convergence close to two.
This is in accordance with the theoretical order of this scheme, first order in
time and second order in space, and the high time-dependency of the solution.
Whereas, with LADER the expected second order of accuracy is achieved. 

\subsection{Manufactured test 2. Turbulent flow with species transport}\label{sec:manufacturedtest2}
The second academic test to be considered is a modification of Test 1
to account for the turbulence and species transport equations.
Let us define the flow as
\begin{gather}
\rho =	1,\\
\pi(x,y,z,t) =	\cos(\pi t (x+y+z)),\\
\mathbf{u}(x,y,z,t) =\left(  \sin(\pi y t)\cos(\pi z t),\;-\cos(\pi z^3 t),\;	\exp(-2\pi x t^2)\right)^T,\\
k(x,y,z,t) =  \sin(\pi x t)+2 ,\\
\varepsilon(x,y,z,t) = \exp(-\pi z t)+1   ,\\
y(x,y,z,t) = \sin(\pi x t)+2 .
\end{gather}
with parameters $\mu=10^{-2}$, $\mathcal{D}=10^{-3}$.
{ For the exact solution to verify the equations, taught expressions
of the source terms have to be taken. They have been included in \ref{sec:appendix_sourceterms}. }

We consider the meshes already defined in Table \ref{tab:mmstest_mesh} and a $\CFL=10$
({ the reason why this large value of CFL is admitted was studied in \cite{BTVC16}}).
Dirichlet boundary conditions  are set for all the equations on the boundary. The computed errors
are presented in Table \ref{mms_turb_err_ord}. The results obtained for CVC-orth confirm that
using only second order in space for computing the flux terms and neglecting the non orthogonal
component will not capture properly the turbulence. Second order in space must also be used to
approximate the diffusion terms and the whole flux should be computed. 
Furthermore, a second order in time scheme {improves} the results and order attained.
\begin{table}\resizebox{\linewidth}{!}{
	\begin{tabular}{|c|c||c|c|c||c|c|}
		\hline Method & Variable & $E_{M_1}$ & $E_{M_2}$ & $E_{M_3}$ & $o_{M_1/M_2}$ & $o_{M_2/M_3}$ \\\hline\hline
		\multirow{5}{*}{Order 1}      &$\pi$                     & $6.97E-01$ & $5.52E-01$ & $4.93E-01$ & $0.34$  & $0.16$\\
		&$\mathbf{w}_{\mathbf{u}}$ & $4.40E-02$ & $2.80E-02$ & $2.18E-02$ & $0.65$  & $0.36$\\
		&$w_{k}$                   & $3.85E-02$ & $2.45E-02$ & $2.09E-02$ & $0.65$  & $0.23$\\
		&$w_{\varepsilon}$         & $1.53E-02$ & $8.47E-03$ & $6.40E-03$ & $0.85$  & $0.40$\\
		&$w_{y}$                   & $2.93E-02$ & $2.04E-02$ & $1.76E-02$ & $0.52$  & $0.21$\\
		\hline
		\multirow{5}{*}{CVC-orth.}  &$\pi$                     & $6.23E-01$ & $5.13E-01$ & $4.75E-01$ & $0.28$ & $0.11$ \\
		&$\mathbf{w}_{\mathbf{u}}$ & $4.08E-02$ & $2.65E-02$ & $2.10E-02$ & $0.62$ & $0.33$ \\
		&$w_{k}$                   & $3.18E-02$ & $2.13E-02$ & $1.95E-02$ & $0.58$ & $0.13$ \\
		&$w_{\varepsilon}$         & $1.51E-02$ & $8.17E-03$ & $6.15E-03$ & $0.89$ & $0.41$ \\
		&$w_{y}$                   & $2.51E-02$ & $1.85E-02$ & $1.68E-02$ & $0.44$ & $0.14$ \\
		\hline                             
		\multirow{5}{*}{CVC-G}      &$\pi$                     & $2.70E-01$ & $7.60E-02$ & $2.09E-02$ & $1.83$  & $1.86$ \\
		&$\mathbf{w}_{\mathbf{u}}$ & $1.50E-02$ & $5.18E-03$ & $1.49E-03$ & $1.54$  & $1.80$ \\
		&$w_{k}$                   & $1.54E-02$ & $3.24E-03$ & $8.22E-04$ & $2.25$  & $1.98$ \\
		&$w_{\varepsilon}$         & $1.06E-02$ & $2.39E-03$ & $6.35E-04$ & $2.15$  & $1.91$ \\
		&$w_{y}$                   & $7.27E-03$ & $1.89E-03$ & $4.86E-04$ & $1.94$  & $1.96$ \\
		\hline
		\multirow{5}{*}{LADER}   &$\pi$                     & $2.68E-01$ & $7.61E-02$ & $2.10E-02$ & $1.82$  & $1.86$ \\
		&$\mathbf{w}_{\mathbf{u}}$ & $1.51E-02$ & $5.17E-03$ & $1.50E-03$ & $1.55$  & $1.79$ \\
		&$w_{k}$                   & $1.37E-02$ & $2.51E-03$ & $5.89E-04$ & $2.45$  & $2.09$ \\
		&$w_{\varepsilon}$         & $9.87E-03$ & $1.80E-03$ & $4.09E-04$ & $2.46$  & $2.14$ \\
		&$w_{y}$                   & $7.25E-03$ & $1.60E-03$ & $3.79E-04$ & $2.18$  & $2.08$ \\
		\hline
	\end{tabular}}
	\caption{ Manufactured test 2. Turbulent flow.  Observed errors and convergence rates. $\CFL=10$.}\label{mms_turb_err_ord}
\end{table}

\subsection{Test 3. Gaussian sphere}
The next problem to be analysed is the Gaussian sphere test introduced in
\cite{BS12} and \cite{Saa11}. We consider a normal distribution function in
the domain $\Omega=\left[-0.9,0.9 \right]\times \left[-0.9,0.9 \right] \times \left[-0.3,0.3 \right]$
with standard deviation $0.08$ and mean $0.25$. The density is one, the velocity vector is defined
as $\mathbf{u}(x,y,z,t) = (-y,\; x,\; 0)^T$ and we assume that the diffusion matrix is given by
$\mathcal{D}=\mu$. Hence, the solution of the problem is given by
\begin{equation}
\mathrm{y}(x,y,z,t) = \left(\frac{\sigma_0}{\sigma(t)}\right)^3 \exp\left(\frac{-r}{2 \sigma(t)^2}\right)
\end{equation}
with
\begin{gather}
r(x,y,z,t) = \left(\bar{x}+0.25\right)^2+\bar{y}^2+z^2,\quad
\sigma(t)=\sqrt{\sigma_0^2+2 t \mathcal{D}},\\
\bar{x}=x \cos(t)+y\sin(t),\quad
\bar{y}=-x \sin(t)+y\cos(t).
\end{gather}
The flow definition is completed setting the source terms 
\begin{equation}
\mathbf{f}_{\mathbf{u}}(x,y,z,t)=(-x,-y,0)^T,\quad f_{\mathrm{y}}(x,y,z,t)=0,
\end{equation}
and considering Dirichlet boundary conditions.

In order to analyse the accuracy in time and space, five structured meshes were generated.
The properties of these meshes can be seen in Table \ref{tab:bell_mesh}, where $h$ denotes
the size of the cubes used to generate the tetrahedra of the finite element mesh.
\begin{table}[H]
	\begin{center}
		\begin{tabular}{|c||c|c|c|c|}
			\hline Mesh  &  Finite elements & Vertices & Nodes &  $h $ \\\hline
			\hline $M_1$ &  $ 11664 $ & $ 2527 $ &  $ 24408 $ &  $ 0.1 $ \\
			\hline $M_2$ &  $ 18522 $ & $ 3872 $ &  $ 38514 $ &   $ 0.0857 $ \\
			\hline $M_3$ &  $ 54000 $ & $ 10571 $ &  $ 111000 $ &   $ 0.06 $ \\
			\hline $M_4$ &  $ 93312 $ & $ 17797 $ &  $ 190944 $ &   $ 0.05 $ \\
			\hline $M_5$ &  $ 182250 $ & $ 33856 $ &  $ 256711 $ &   $ 0.04 $ \\
			\hline
		\end{tabular}
		\caption{Test 3. Gaussian sphere. Mesh features.}
		\label{tab:bell_mesh}
	\end{center}
\end{table} 

{Table \ref{tab:bell_err_1e3_LADER}  shows} the results obtained for the test
considering $\mu=10^{-3}$. On the other hand, in {Table \ref{tab:bell_err_1e2_LADER}}
the errors and orders of accuracy for $\mu=10^{-2}$ are presented. In both test cases we have
assumed a final time $t_{\textrm{end}}=2\pi$ so that the sphere completes one revolution.
Two different methodologies were considered to run these tests: CVC-G
and LADER. 
We can observe that for $\mu=10^{-3}$ CVC-G scheme only achieves first order and for $\mu=10^{-2}$
the order obtained is a bit greater but still lower than two for the velocities approach.
Meanwhile, using LADER we obtain the expected second order in both tests cases and the errors
obtained decrease. These improvements are due to considering a second order method in time.
The high diffusivity of the test makes necessary to consider second order in both, time and space,
to achieve good approaches for all the unknowns of the problem.

The previous discussion is also consistent with the graphical results presented in Figures \ref{fig:contours_y_c4_e3_cl}-\ref{fig:c4_e2_curves}.

\begin{table}[H]
		\renewcommand{\arraystretch}{1.5}
	\begin{center}
		\resizebox{\textwidth}{!}{\begin{tabular}{|c||c||c|c||c|c||c|c|}
			\cline{3-8} \multicolumn{2}{c|}{} & \multicolumn{2}{c||}{$\pi$}  & \multicolumn{2}{c||}{ $\mathbf{w} _{\mathbf{u} }$} & \multicolumn{2}{c|}{$w_y$}\\\cline{3-8}
			\multicolumn{2}{c|}{} & $E_{M_i}$ & $o_{M_{i-1}/M_{i}}$ & $E_{M_i}$ & $o_{M_{i-1}/M_{i}}$ & $E_{M_i}$ & $o_{M_{i-1}/M_{i}}$ \\  \hline
			\multirow{5}{1.5cm}{CVC-G}&$M_1$ & $7.48E-02$& &$1.33E-01$ & &$4.00E-02$ & \\\cline{2-8}
			&$M_2$ & $6.59E-02$&$0.82$ &$1.17E-01$ &$0.84$ &$3.56E-02$ &$0.75$ \\\cline{2-8}
			&$M_3$ & $4.75E-02$&$0.92$ &$8.52E-02$  &$0.89$  &$2.63E-02$ & $0.85$\\\cline{2-8}
			&$M_4$ & $4.02E-02$&$0.92$ &$7.21E-02$  &$0.91$ &$2.20E-02$ & $0.99$\\\cline{2-8}
			&$M_5$ & $3.24E-02$&$0.96$ &$5.82E-02$ &$0.96$ &$1.73E-02$ &$1.08$ \\\hline \hline
			\multirow{5}{1.5cm}{LADER} &$M_1$ & $1.02E-03$& &$2.48E-03$ & &$2.31E-02$ & \\\cline{2-8}
			&$M_2$ & $7.25E-04$&$2.19$ &$1.81E-03$ &$2.02$ &$1.83E-02$ &$1.50$ \\\cline{2-8}
			&$M_3$ & $3.23E-04$&$2.27$ &$8.17E-04$  &$2.24$  &$1.00E-02$ & $1.70$\\\cline{2-8}
			&$M_4$ & $2.12E-04$&$2.32$ &$5.41E-04$  &$2.26$ &$7.11E-03$ & $1.88$\\\cline{2-8}
			&$M_5$ & $1.25E-04 $&$2.36$&$3.24E-04 $ &$2.30$ & $4.59E-03 $ &$1.97$\\\hline
		\end{tabular}}
			\caption{Test 3. Gaussian sphere, $\mu=10^{-3}$. Observed errors and convergence rates.  $\CFL=5$ for CVC-G and $\CFL=0.5$ for LADER.}
		\label{tab:bell_err_1e3_LADER}
	\end{center}
\end{table}

\begin{table}[H]
		\renewcommand{\arraystretch}{1.5}
	\begin{center}
		\resizebox{\textwidth}{!}{\begin{tabular}{|c||c||c|c||c|c||c|c|}
				\cline{3-8} \multicolumn{2}{c|}{} & \multicolumn{2}{c||}{$\pi$}  & \multicolumn{2}{c||}{ $\mathbf{w} _{\mathbf{u} }$} & \multicolumn{2}{c|}{$w_y$}\\\cline{3-8}
				\multicolumn{2}{c|}{} & $E_{M_i}$ & $o_{M_{i-1}/M_{i}}$ & $E_{M_i}$ & $o_{M_{i-1}/M_{i}}$ & $E_{M_i}$ & $o_{M_{i-1}/M_{i}}$ \\  \hline
				\multirow{5}{1.5cm}{CVC-G}&$M_1$ & $3.31E-02$& &$5.13E-02$ & &$2.19E-03$ & \\\cline{2-8}
				&$M_2$ & $2.74E-02$&$1.22$ &$4.29E-02$ &$1.16$ &$1.67E-03$ &$1.77$ \\\cline{2-8}
				&$M_3$ & $1.72E-02$&$1.31$ &$2.74E-02$  &$1.26$  &$9.02E-04$ & $1.73$\\\cline{2-8}
				&$M_4$ & $1.34E-02$&$1.39$ &$2.15E-02$  &$1.34$ &$6.62E-04$ & $1.70$\\\cline{2-8}
				&$M_5$ & $9.69E-03$&$1.44$ &$1.57E-02$ &$1.40$ &$4.54E-04$ &$1.69$ \\\hline \hline
				\multirow{5}{1.5cm}{LADER} &$M_1$ & $3.68E-04$& &$9.02E-04$ & &$1.61E-03$ & \\\cline{2-8}
				&$M_2$ & $2.54E-04$&$2.41$ &$6.24E-04$ &$2.02$ &$1.16E-03$ &$2.12$ \\\cline{2-8}
				&$M_3$ & $1.12E-04$&$2.29$ &$2.60E-04$  &$2.24$  &$5.55E-04$ & $2.07$\\\cline{2-8}
				&$M_4$ & $7.57E-05$&$2.15$ &$1.66E-04$  &$2.26$ &$3.85E-04$ & $2.01$\\\cline{2-8}
				&$M_5$ & $4.78E-05$&$2.06$ &$9.54E-05 $ &$2.30$ & $	2.48E-04$ &$1.97$\\\hline
		\end{tabular}}
		\caption{Test 3. Gaussian sphere, $\mu=10^{-2}$. Observed errors and convergence rates.
			$\CFL=5$ for CVC-G and $\CFL=0.5$ for LADER.}
		\label{tab:bell_err_1e2_LADER}
	\end{center}
\end{table}

\begin{figure}[H]
	\centering
	\includegraphics[width=\linewidth]{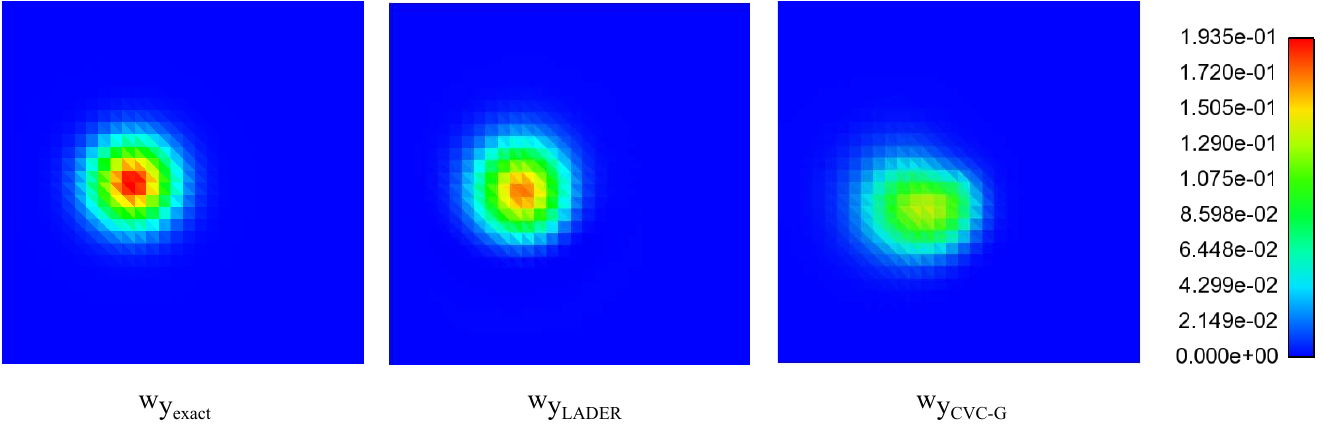}
	 \caption{{Test 3. Gaussian sphere, $\mu=10^{-3}$. Contours of $w_y$ at plane $z=0$ using Mesh $M_3$.
		Left: exact solution. Centre:  LADER. Right: CVC-G.}}
	\label{fig:contours_y_c4_e3_cl}
\end{figure}
\begin{figure}[H]
	\centering
	\includegraphics[width=\linewidth]{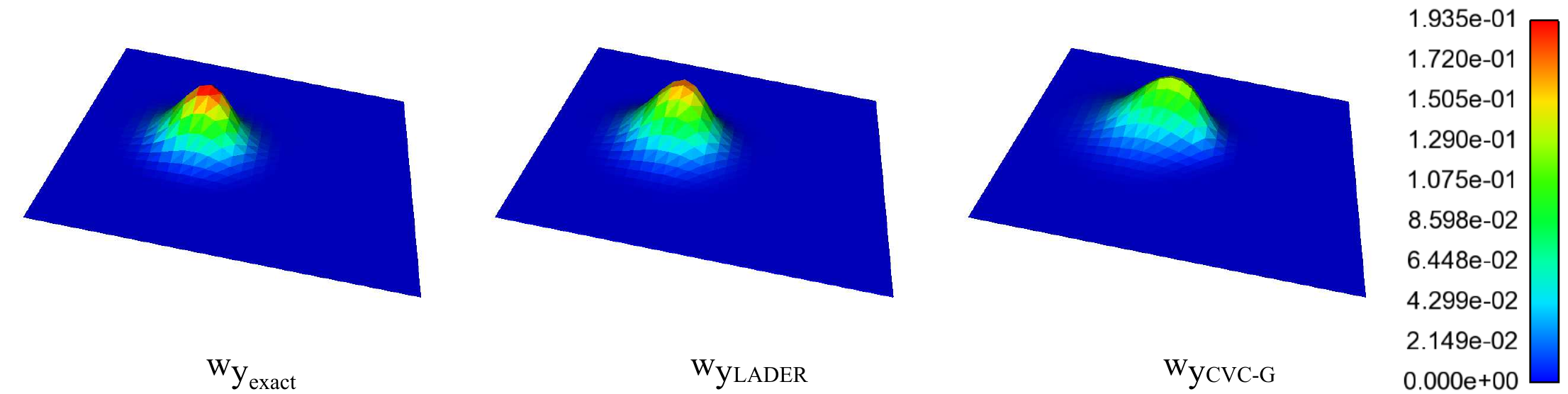}
	\caption{{ Test 3. Gaussian sphere, $\mu=10^{-3}$. Elevated surfaces of $w_y$ at plane $z=0$ using Mesh $M_3$.
		Left: exact solution. Centred: LADER. Right: CVC-G.}}
	\label{fig:elevated_y_c4_e3_cl}
\end{figure}
\begin{figure}[H]
	\centering
	\includegraphics[width=\linewidth]{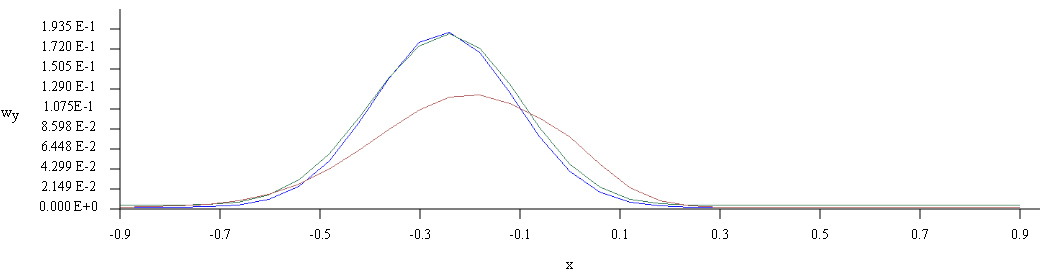}
	\caption{Test 3. Gaussian sphere, $\mu=10^{-3}$. Profile of the exact solution (blue) and the computed
		solutions using LADER (green) and CVC-G (red) at plane $y=0$.}
	\label{fig:campana43_prof}
\end{figure}

\begin{figure}[H]
	\centering
	\includegraphics[width=\linewidth]{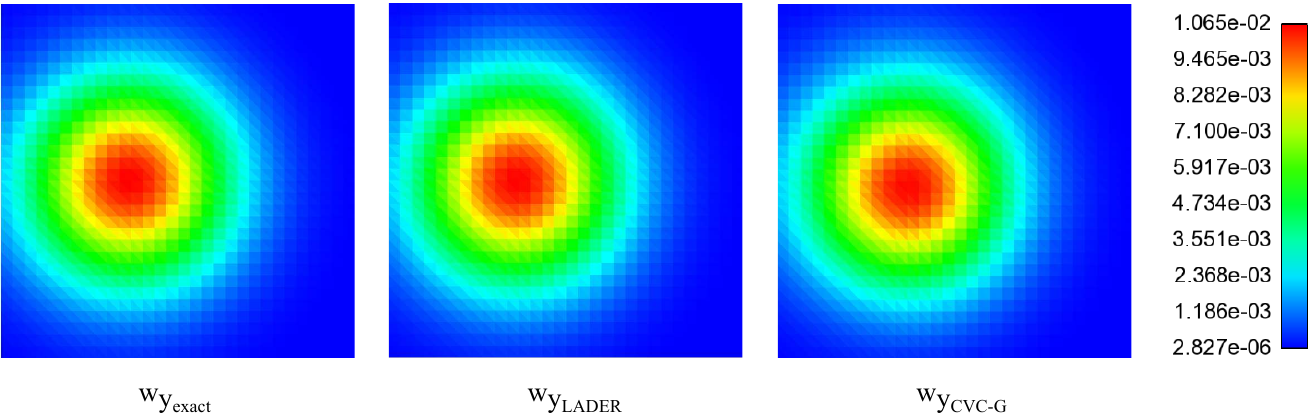}
	\caption{{ Test 3. Gaussian sphere, $\mu=10^{-2}$. Contours of $w_y$ at plane $z=0$ using Mesh $M_3$.
		Left: exact solution. Centre:  LADER. Right: CVC-G.}}
	\label{fig:contours_y_c4_e2_cl}
\end{figure}
\begin{figure}[H]
	\centering
	\includegraphics[width=\linewidth]{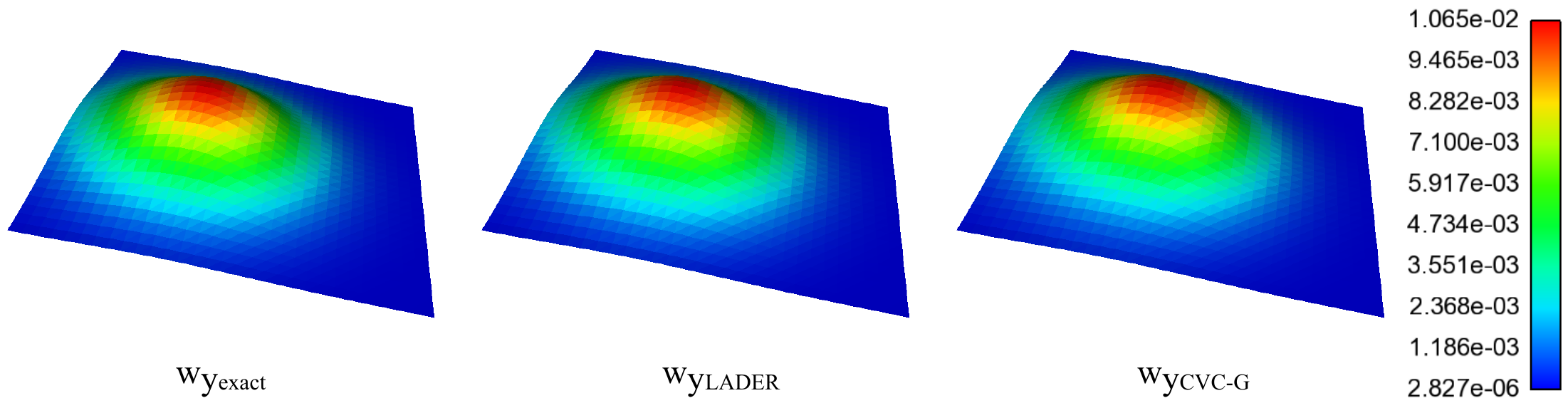}
	\caption{{ Test 3. Gaussian sphere, $\mu=10^{-2}$. Elevated surfaces of $w_y$ at plane $z=0$ using Mesh $M_3$.
		Left: exact solution. Centre:  LADER. Right: CVC-G.}}
	\label{fig:elevated_y_c4_e2_cl}
\end{figure}
\begin{figure}[H]
	\centering
	\includegraphics[width=\linewidth]{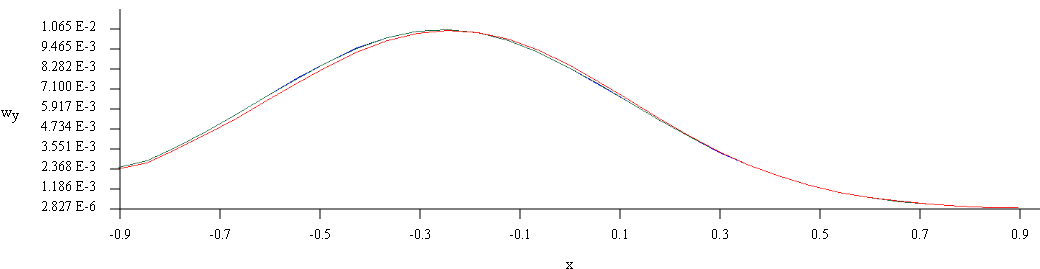}
	\caption{Test 3. Gaussian sphere, $\mu=10^{-2}$. Profile of the exact solution (blue) and the computed
		solutions using LADER (green) and CVC-G (red) at plane $y=0$.}
	\label{fig:c4_e2_curves}
\end{figure}

\subsection{Test 4. Flow around a cylinder}
We consider the steady-state problem of a flow around a cylinder which has been introduced in \cite{STDKR96}
and employed, for instance, in \cite{BFSV14} and \cite{Vol02} as a benchmark problem.
The computational domain consists of a solid cylinder surrounded by a rectangular channel
in which the flow evolves (see Figure \ref{fig:dominio}). The dynamic viscosity of the fluid is
$\mu=10^{-3}$ and the inlet velocity has the form
\begin{equation}
\mathbf{u}(x,y,z,t)=\left(16Uyz(H-y)(H-z)/H^4, 0, 0\right)^T,
\end{equation}
with $U=0.45$, $H=0.41$. Based on the viscosity, the cylinder diameter, $D=0.1$, and an estimate
of $0.2$ for the mean inflow velocity, the flow has a Reynolds number of $20$.
At the outlet Neumann boundary conditions are considered.
\begin{figure}
	\centering
	\includegraphics[width=0.7\linewidth]{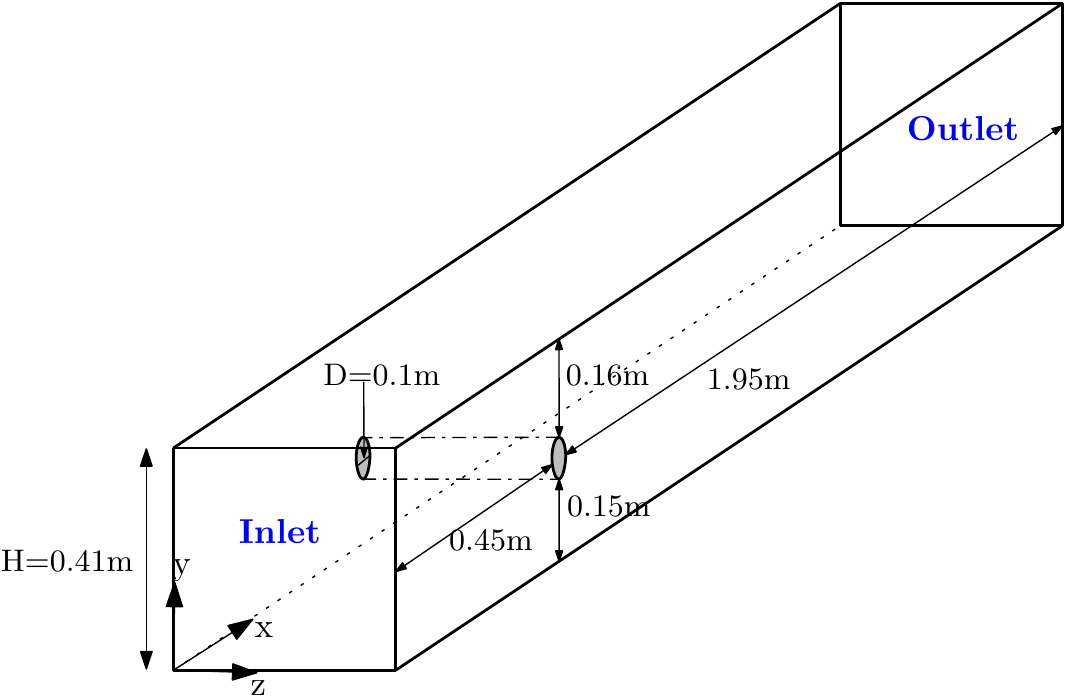}
	\caption{Test 4. Flow around a cylinder. Geometry.}
	\label{fig:dominio}
\end{figure}
The mesh employed to obtain the numerical solution consists of
$449746$ finite elements and $909004$ finite volumes.

The drag and lift coefficients for this problem are expressed by
\begin{equation}
c_d=\frac{500}{0.41}F_d, \text{\ } c_l=\frac{500}{0.41}F_l,
\end{equation}
where $F_d$ and $F_l$ are the drag and lift forces, respectively:
\begin{equation}
F_d=\int_S \left(\mu \frac{\partial \mathbf{u}_{\tau}}{\partial
	\mathbf{n}_S} n_y -\pi n_x \right) \mathrm{dS}, \quad F_l=\int_S
\left(-\mu \frac{\partial \mathbf{u}_{\tau}}{\partial
	\mathbf{n}_S }n_x-\pi n_y \right) \mathrm{dS} \label{eq:fdfl}
\end{equation}
with $ \mathbf{n}_S=(n_x, n_y,0)^t$ the inward pointing unit normal
with respect to  $\Omega$, $S$ the surface of the cylinder, and
$\mathbf{n}_{\tau}=(n_y, -n_x,0)^t$ one of the tangential
vectors, the other one being $(0,0,1)^T$. The drag and lift forces were computed following \cite{BFSV14}. 
As convergence criterion, at iteration $k$,  we consider,
\begin{equation}
\frac{1}{\Delta t} \|\mathbf{W}^{k}_{M}-\mathbf{W}^{k-1}_{M} \|_{L^{\infty}(\Omega)^3}\le 10^{-4} .
\end{equation}

Four different simulations regarding the method employed were run:

\vspace*{0.3cm}
\begin{minipage}{0.11\textwidth}
\end{minipage}\hfill\begin{minipage}{0.88\textwidth}
	\begin{itemize}
		\item[Method 1:] the first order method presented in \cite{BFSV14}, which considers
		the Rusanov scheme as the numerical flux{.}	\item[Method 2:] the second
		order in space and first order in time CVC-orth{.}
		\item[Method 3:] CVC-G, also second order in space and first order in time{.}
		\item[Method 4:] the second order method given by LADER.
	\end{itemize}
\end{minipage}

\vspace*{0.3cm}
\noindent 
The numerical results are summarized in Table \ref{tab:cilindro}.
\begin{table}
	\begin{center}
		\renewcommand{\arraystretch}{1.2}
		\begin{tabular}{|l|c|c|c|c|}
			\hline \multirow{3}{80pt}{Method}  & \multirow{3}{50pt}{Time iterations}  & $C_D$ & $C_L$ & $D\pi$  \\			  
			&  &   (min,max)   & (min,max) & (min,max)\\
			& &    $( 6.05, 6.25)$ & $(0.008, 0.01)$ & $(0.165, 0.175)$\\\hline
			1. Order 1    &  $1745$  & $6.79$   & $0.0062$  & $0.1656$\\
			2. CVC-orth &  $72442$ & $6.2463$ & $-0.00067$ & $0.1651$ \\ 
			3. CVC-G    &  $73994$ & $6.1619$ & $0.01996$ & $0.1616$\\
			4. LADER &  $85638$ & $6.1249$ & $0.0161$  & $0.1662$ \\
			\hline
		\end{tabular}
		\caption{Test 4. Flow around a cylinder. Obtained values for the aerodynamic
			coefficients and the pressure difference.}\label{tab:cilindro}
	\end{center}
\end{table} 
Along with the aerodynamic coefficients, the pressure difference $D \pi$
between the points $\mathbf{p}_1=(0.45, 0.2, 0.205)$ and $\mathbf{p}_2=(0.55, 0.2, 0.205)$,
has been computed.  We observe that the solutions obtained with the higher order method,
as expected theoretically, are the most accurate { with} respect to the
reference intervals obtained from the experimental data on \cite{STDKR96}. 
Finally, Figures \ref{fig:pressurecptiempos}, \ref{fig:velocitycptiempos} and
\ref{fig:pathlinescptiempos} show the results obtained using LADER methodology.

\begin{figure}[H]
	\centering
	\includegraphics[width=\linewidth]{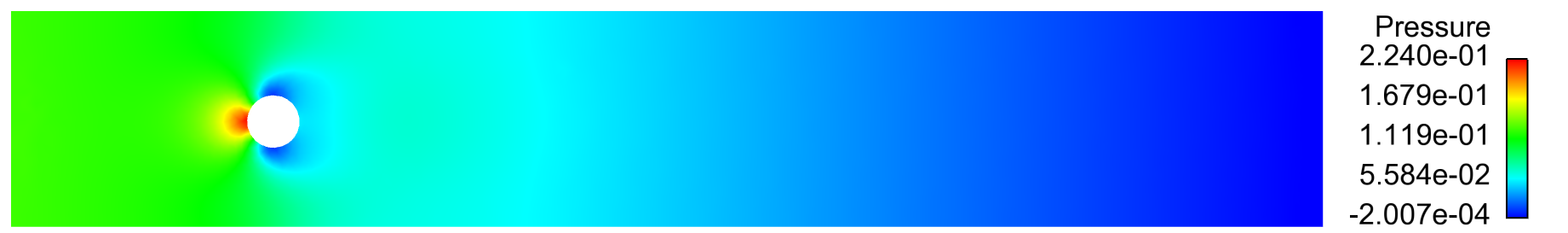}
	\caption{{Test 4. Flow around a cylinder. Pressure on $z=0.205$.}}\label{fig:pressurecptiempos}
	
\end{figure}
\begin{figure}[H]
	\centering
	\includegraphics[width=\linewidth]{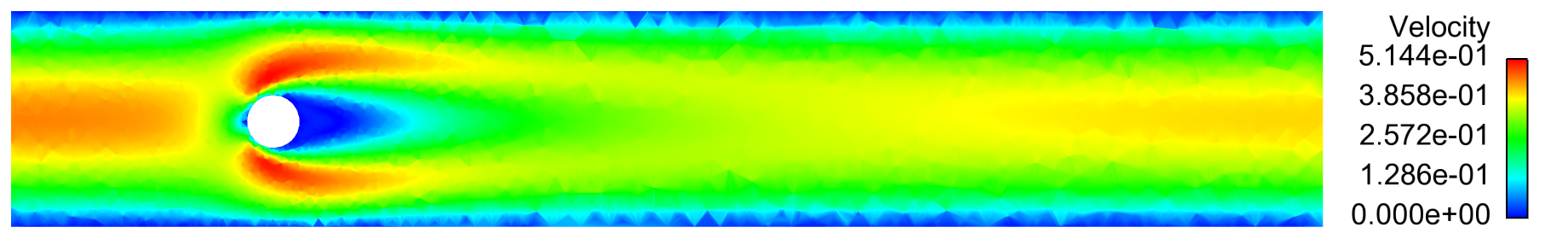}
	\caption{{Test 4. Flow around a cylinder. Velocity  magnitude on $z=0.205$.}}
	\label{fig:velocitycptiempos}
\end{figure}
\begin{figure}[H]
	\centering
	\includegraphics[width=\linewidth]{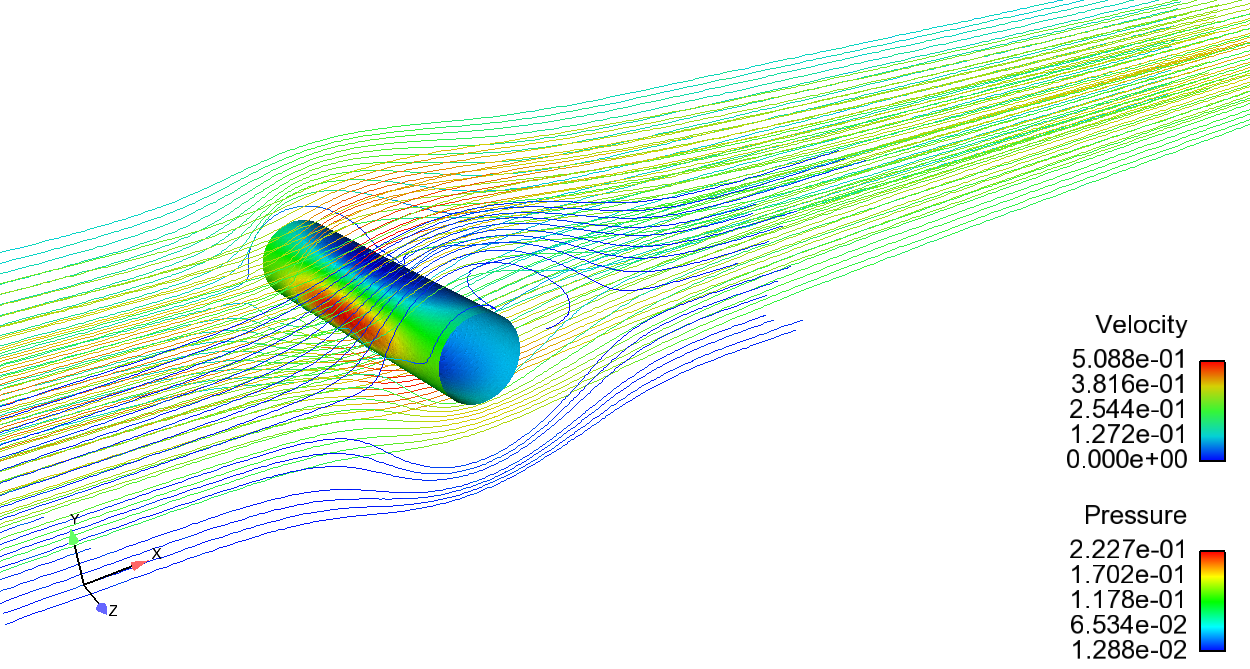}
	\caption{Test 4. Flow around a cylinder. Pressure over the cylinder and streamlines.}
	\label{fig:pathlinescptiempos}
\end{figure}

\section{Summary and conclusions}
In this paper a  projection hybrid high order FV/FE  method for incompressible flows has been presented.
Navier-Stokes equations have been coupled with the $k-\varepsilon$ model in order to simulate turbulent flows.
The system to be solved was enlarged with respect to \cite{BFSV14} considering species transport.
High order of accuracy has shown necessary for the proper computation of turbulent effects.
Two different methodologies to achieve second order were presented. Firstly, CVC Kolgan provided a
second order in space and first order in time scheme. To attain second order in space and time a
new method was proposed: LADER. The corresponding accuracy and stability analysis were presented for
the advection-diffusion-reaction equation. Godunov's theorem was circumvented thanks to an ENO-based approach.
The computation of the gradients involved { on diffusion} terms was done via Galerkin.
The method was applied to manufactured test problems in order to asses the accuracy. Furthermore,
different benchmarks were considered and the results obtained were successfully confronted with experimental data.

\section*{{ Acknowledgements}}
This work was financially supported by Spanish MICINN projects
MTM2008-02483,
CGL2011-28499-C03-01
and MTM2013-43745-R;
by the Spanish MECD under grant FPU13/00279; 
by the Xunta de Galicia Conseller\'ia de Cultura Educaci\'on
e Ordenaci\'on Universitaria under grant Axudas de apoio
\'a etapa predoutoral do Plan I2C; by Xunta de Galicia and FEDER under research project GRC2013-014
and 
by Fundaci\'on Barri\'e under grant \textit{Becas de posgrado en el extranjero}.

\bibliographystyle{plain}
\bibliography{./mibiblio}

\appendix
\section{LADER}\label{sec:appendix}
ADER methodology was successfully extended in \cite{BTVC16} to solve
advection-diffusion-reaction equations. The developed method, ADER-ADRE,
is of second order in space and time and stability can be ensured by determining
the time step taking into account the advection, diffusion and reaction coefficients.
Despite this method is easily programmed for the one dimensional code, the computation
of the fluxes focusing on a particular finite volume at each time is not suitable for
the mesh structure we have in the three-dimensional {case}. In this case, we would like
to take profit from the loop on the faces of the finite volumes and reduce the
computational cost. To do that, the LADER method, which preserves the second order
and the stability of the ADER-ADRE method, was developed.

Let us consider the advection-diffusion-reaction equation
\begin{equation}
\partial_t q(x,t) +\lambda \partial_x q(x,t) =\partial_{x} \left(\alpha(x,t) \partial_x q(x,t)\right) +\beta  q(x,t)  \label{ADRE}
\end{equation}
where 
\begin{itemize}
	\item $q(x,t)$ is the conservative variable,
	\item $x,t$ are the spatial and temporal independent variables,
	\item  $\lambda$ is the characteristic speed,
	\item $\alpha(x,t)$ is the diffusion coefficient, a prescribed function,
	\item $\beta$ is the coefficient of the reaction term. 
\end{itemize}
Then, LADER method is divided into the following steps:
\begin{description}
	\item[Step 1.] Polynomial reconstruction.	
	We consider a reconstruction of the data in terms of piecewise first-degree polynomials of the form
	\begin{equation}
	p_{i}(x)=\left\lbrace \begin{array}{lr}
	p_{i\,L}(x)=q_i^n+\Delta_{i\,L}^n (x-x_i),  & x\in\left(x_{i-\frac{1}{2}},x_{i}\right],\\[10pt]
	p_{i\,R}(x)=q_i^n+\Delta_{i\,R}^n (x-x_i), & x\in\left[x_{i},x_{i+\frac{1}{2}}\right),\end{array}	
	\right.
	\end{equation}
	where $\Delta_{i\,L}^n$, $\Delta_{i\,R}^n$ denote the approximations of the spatial
	derivatives of $q(x,t)$ at time $t^{n}$ related to two auxiliary elements of volume
	$\mathcal{C}_i= 	\left[x_{i-\frac{1}{2}},x_{i+\frac{1}{2}}\right]$: 
	\begin{equation}
	T_{i-1 i \, L}=\left[x_i,x_{i+1}\right], \qquad
	T_{i i+1 \, R}=\left[x_{i-1},x_{i}\right]
	\end{equation}
	(see Figure \ref{fig:ader_local_malla}).
	
	\item[Step 2.] \label{grp} Solution of the generalized Riemann problem (GRP).	
	To construct the numerical flux at $x_{i+\frac{1}{2}}$ the following generalizations
	of the Classical Riemann Problem are made. On the one hand, the initial condition is
	assumed to be a piecewise first-degree polynomial. On the other hand, the partial
	differential equation accounts for the diffusion and reaction terms. That leads to the problem
	\begin{equation}
	\left\lbrace \begin{array}{l}
	\partial_t q\left(x,t\right)+\lambda \partial_x q\left(x,t\right) = \partial_x\left( \alpha\partial_x q\right)\left(x,t\right)+\beta q\left(x,t\right),\\
	q(x,0)=\left\lbrace \begin{array}{ll}
	p_{i\, R}(x), & x<0,\\
	p_{i+1\, L}(x), & x>0.
	\end{array} \right.
	\end{array}\right.\label{eq:GRP}
	\end{equation}
	
	\item[Step 3.] \label{adre_difusion} \label{ader_source} Diffusion and reaction terms.	
	These terms are computed by approximating the integrals by the mid-point rule in both space and time.	
\end{description} 

\begin{figure}
	\centering
	\includegraphics[width=\linewidth]{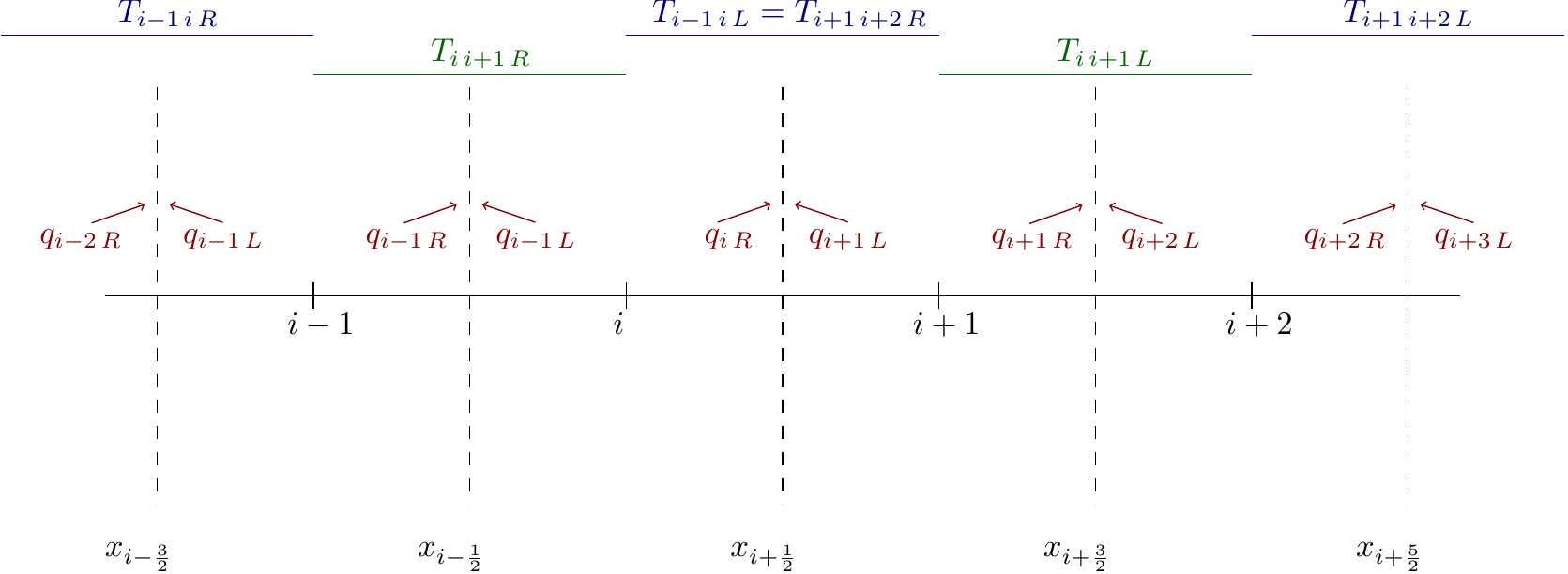}
	\caption{Mesh and nomenclature.}
	\label{fig:ader_local_malla}
\end{figure}

The solution of the GRP at the interface $x_{i+\frac{1}{2}}$, set in Step 2,
is expressed as a Taylor series expansion in time, namely,
\begin{equation}
\overline{q}_{i+\frac{1}{2}}=q(0,0_{+})+\tau\partial_t q(0,0_{+}).
\end{equation}
The first term of the above equation is computed as the solution of the classical Rieman problem
\begin{equation}
\left\lbrace \begin{array}{l}
\partial_t q\left(x,t\right)+\lambda \partial_x q\left(x,t\right) =0,\\
q(x,0)=\left\lbrace \begin{array}{ll}
q_{i\, R}, & x<0,\\
q_{i+1\, L} , & x>0,
\end{array} \right.
\end{array}\right.
\end{equation} 
where
\begin{eqnarray}
q_{i\, R}=q_{i} + \frac{\Delta x}{2}\Delta_{i\, R}^{n} = q_{i}+ \frac{q_{i}-q_{i-1}}{2} =  \frac{1}{2} \left(3q_{i}-q_{i-1}\right), \\
q_{i+1\, L}=q_{i+1} - \Delta_{i+1\, L}^{n} = q_{i+1}- \frac{q_{i+2}-q_{i+1}}{2}=  \frac{1}{2} \left(3q_{i+1}-q_{i+2}\right).
\end{eqnarray}
Then, 
\begin{equation*}
q\left( 0,0_{+}\right) =\left\lbrace
\begin{array}{lr}
q_{i\, R}, & \lambda>0,\\
q_{i+1\, L},& \lambda<0.
\end{array}
\right.
\end{equation*}

The second term is computed following the Cauchy-Kovalevskaya procedure which allows us
to express the time derivative of the conservative variable as a combination of the spatial derivatives,
\begin{equation}
\partial_t q(x,t)=-\lambda \partial_x q(x,t) +\partial_x\left( \alpha\partial_x q\right)\left(x,t\right)+\beta q\left(x,t\right),
\end{equation}
so,
\begin{equation}\overline{q}_{i+\frac{1}{2}} = q(0,0_+) + \tau\left[-\lambda \partial_x q(0,0_+) + \partial_x\left( \alpha\partial_x q\right)  (0,0_+)  + \beta q(0,0_+)\right].\end{equation}
For now on, we will focus on the scheme for $\lambda>0$ (the case $\lambda<0$ is analogous). Approximating 
\begin{eqnarray}
\partial_x q(0,0_+) =\Delta_{i+\frac{1}{2}}^n\approx \frac{1}{\Delta x}\left(q_{i+1}^n-q_{i}^n\right),\\
\partial_x\left(\alpha\partial_x q\right)\left(0,0_+\right)\ =
\left( \Delta\alpha\Delta\right)_{i+\frac{1}{2}}^n
\notag\\
\approx \frac{1}{\Delta x^2} \left[ \alpha_{i+1}^n\left(q_{i+2}^n-q_{i+1}^n \right) -\alpha_{i}^n\left(q_{i}^n-q_{i-1}^n \right)  \right]
\end{eqnarray}
and performing exact integration, the numerical flux reads
\begin{gather}
f_{i+\frac{1}{2}}^{n}
= \lambda \overline{q}_{i-\frac{1}{2}^{n}} 
=\lambda \left\lbrace  q_{i\, R}^n + \frac{\Delta t}{2} \left[ - \lambda \Delta_{i+\frac{1}{2}}^n + \left(\Delta\alpha\Delta\right)_{i+\frac{1}{2}}^n
+\beta q_{i}^n \right]\right\rbrace \notag \\
=\lambda \left\lbrace  q_{i}^n+\frac{1}{2}  \left( q_{i}^n-q_{i-1}^n\right) - \frac{\lambda \Delta t}{2\Delta x}\left(  q_{i+1}^n-q_{i}^n\right)
\right.\notag\\\left.+\frac{\Delta t}{2\Delta x^2}\left[\alpha_{i+1}^{n}\left(q_{i+2}^{n}-q_{i+1}^{n}\right)
-\alpha_{i}^{n}\left(q_{i}^{n}-q_{i-1}^{n}\right)\right]
+\beta \frac{\Delta t}{2} q_{i}^n
\right\rbrace .
\end{gather}

For the diffusion and reaction terms computation we follow \cite{BTVC16}. {We consider the centred slopes
\begin{gather}
\Delta_i^{n}=\frac{q_{i+1}^n-q_{i-1}^{n}}{2\Delta x},\\
\left(\Delta\alpha\Delta\right)_{i}^{n}=\frac{\alpha^{n}_{i+\frac{1}{2}}\left( q_{i+1}^{n}-q_{i}^{n}\right)-\alpha^{n}_{i-\frac{1}{2}}\left( q_{i}^{n}-q_{i-1}^{n}\right) } {\Delta x^2}
\end{gather}
 and the upwind slope
\begin{equation}
	\breve{\Delta}_{i+\frac{1}{2}}^{n} = q_{i+1}^{n}-q_{i}^{n}.
\end{equation}
Then, the evolved values of the  diffusion and reaction terms read
\begin{gather}
\overline{\left(\Delta \alpha \Delta\right)_i^{n}}=
\frac{\overline{\alpha^n_{i+\frac{1}{2}}} \, \overline{\Delta_{i+\frac{1}{2}}}- \overline{\alpha^n_{i-\frac{1}{2}}}\, \overline{\Delta_{i-\frac{1}{2}}}}{\Delta x} \notag
=
\frac{1
}{\Delta x^2} \left( \overline{\alpha^n_{i+\frac{1}{2}}} \overline{\breve{\Delta}_{i+\frac{1}{2}}^{n}}
-\overline{\alpha^n_{i-\frac{1}{2}}}  \overline{\breve{\Delta}_{i-\frac{1}{2}}^{n}} \right)\notag\\
= \! \frac{1}{\Delta x^2} \left\lbrace \!
\left[ \alpha^n_{i+\frac{1}{2}} \! +\! \frac{\Delta t}{2} \partial_t \alpha_{i+\frac{1}{2}}^{n}\right] \!\! \left[ \!\phantom{\frac{a}{a}\!\!\!\!}
\breve{\Delta}_{i+\frac{1}{2}}^{n}
\!+ \! \frac{\Delta t}{2} \left(
\left(\Delta\alpha\Delta\right)_{i+1}^{n} \!-\! \left(\Delta\alpha\Delta\right)_{i}^{n}\! +\! \beta \breve{\Delta}_{i+\frac{1}{2}}^{n} \right) 
\right] 
\right. \notag\\ \left.
+\!\left[ \alpha^n_{i-\frac{1}{2}}\! +\! \frac{\Delta t}{2}  \partial_t \alpha_{i-\frac{1}{2}}^{n} \right]\!\!\! \left[\!\!\phantom{\frac{a}{a}\!\!\!\!}
-\!\breve{\Delta}_{i-\frac{1}{2}}^{n}
\!+\! \frac{\Delta t}{2} 
\left(\! \! \phantom{\frac{a}{a}\!\!\!\!} \left(\Delta\alpha\Delta\right)_{i-1}^{n}\! -\! \left(\Delta\alpha\Delta\right)_{i}^{n}\! -\!\beta\breve{\Delta}_{i-\frac{1}{2}}^{n} 
\phantom{\frac{a}{a}\!\!\!\!}\!\right) 
\right] \!
\right\rbrace\!,  \label{eq:diffusion_adre}
\end{gather}
\begin{equation}
\beta \overline{q_i^{n}} = \beta\left[ q_i^{n}+\frac{\Delta t}{2}\left(-\lambda \Delta_i^{n} +\left( \Delta\alpha \Delta\right)_i^{n}+\beta q_i^{n}\right) \right] \label{eq:source_adre}\end{equation}
with $\overline{\alpha^n_{i+\frac{1}{2}}}$ and $\overline{\alpha^n_{i-\frac{1}{2}}}$ approximated likewise in \cite{BTVC16}.

\begin{remark2}
	It is important to notice that the evolution of the diffusion term does not account for the advection term. The local treatment proposed in LADER scheme produce a evolved flux which already contains the whole contribution of the assembling of advection and diffusion terms. Hence, the second order of accuracy in space and time will be attained only if we neglect the advection term in the computation of the evolved diffusion.
\end{remark2}
}

Finally, denoting $c=\frac{\lambda \Delta t}{\Delta x}$ the Courant number and $r=\beta \Delta t$ the reaction number, the finite volume scheme for the advection-diffusion-reaction equation results
{
\begin{gather}
q_{i}^{n+1}=q_{i}^{n}-c \left\lbrace \breve{\Delta}_{i-\frac{1}{2}}+\frac{1}{2} \breve{\Delta}_{i-\frac{1}{2}}^{n} - \frac{c}{2}\breve{\Delta}_{i+\frac{1}{2}}^{n}
+\frac{\Delta t}{2\Delta x^2}\left[\alpha_{i+1}^{n}\breve{\Delta}_{i+\frac{3}{2}}^{n}
-\alpha_{i}^{n}\breve{\Delta}_{i-\frac{1}{2}}^{n}\right]
+ \frac{r}{2} \breve{\Delta}_{i-\frac{1}{2}}
\right. \notag\\\left.
-\frac{1}{2}  \breve{\Delta}_{i-\frac{3}{2}}^{n} + \frac{c}{2}\breve{\Delta}_{i-\frac{1}{2}}^{n}
-\frac{\Delta t}{2\Delta x^2}\left[\alpha_{i}^{n}  \breve{\Delta}_{i+\frac{1}{2}}^{n}
-\alpha_{i-1}^{n}\breve{\Delta}_{i-\frac{3}{2}}^{n}\right]
\right\rbrace + \frac{\Delta t}{\Delta x^2} \left\lbrace
\left[ \alpha^n_{i+\frac{1}{2}} + \frac{\Delta t}{2} \partial_t \alpha_{i+\frac{1}{2}}^{n}\right] \right. \notag\\\left.
 \left[ 
\breve{\Delta}_{i+\frac{1}{2}}^{n}
+ \frac{\Delta t}{2\Delta x^2}\left( \alpha^{n}_{i+\frac{3}{2}}\breve{\Delta}_{i+\frac{3}{2}}^{n}-  2\alpha^{n}_{i+\frac{1}{2}}\breve{\Delta}_{i+\frac{1}{2}}^{n} + \alpha^{n}_{i-\frac{1}{2}}\breve{\Delta}_{i-\frac{1}{2}}^{n}\right)  + \frac{r}{2}\breve{\Delta}_{i+\frac{1}{2}}^{n}
\right]
\right. \notag\\
\left. 
 +\left[ \alpha^n_{i-\frac{1}{2}} + \frac{\Delta t}{2}  \partial_t \alpha_{i-\frac{1}{2}}^{n} \right] \left[
-\breve{\Delta}_{i-\frac{1}{2}}^{n}
+ \frac{\Delta t}{2\Delta x^2}\left(-\alpha^{n}_{i+\frac{1}{2}}\breve{\Delta}_{i+\frac{1}{2}}^{n}+ 2\alpha^{n}_{i-\frac{1}{2}}\breve{\Delta}_{i-\frac{1}{2}}^{n}-\alpha^{n}_{i-\frac{3}{2}}\breve{\Delta}_{i-\frac{3}{2}}^{n} \right) 
\right. \right. \notag\\\left. \left. 
- \frac{r}{2}\breve{\Delta}_{i-\frac{3}{2}}^{n}
\right]\right\rbrace +r\left[ q_i^{n}-\frac{c}{4}\left( q_{i+1}^n-q_{i-1}^{n}\right)
 +\frac{\Delta t}{2\Delta x^2}\left( \alpha^{n}_{i+\frac{1}{2}}\breve{\Delta}_{i+\frac{1}{2}}^{n}-\alpha^{n}_{i-\frac{1}{2}}\breve{\Delta}_{i-\frac{1}{2}}^{n} 
\right)+\frac{r}{2} q_i^{n} \right].
\label{eq:local_ader3}
\end{gather}}

\begin{remark2}
	The scheme for the advection-diffusion-reaction equation with constant diffusion coefficient reads
{
\begin{gather}
q_{i}^{n+1}=q_{i}^{n}-c \left[ \breve{\Delta}_{i-\frac{1}{2}}^{n} +\frac{1}{2}  \breve{\Delta}_{i-1}^{n}  - \frac{c}{2}\breve{\Delta}_{i}^{n} 
+\frac{d}{2}\left(\breve{\Delta}_{i+1}^{n}-\breve{\Delta}_{i-1}^{n} \right)
+ \frac{r}{2} \breve{\Delta}_{i-\frac{1}{2}}^{n} 
\right] \notag\\
+ d \left[ 	\breve{\Delta}_{i}^{n}  + \frac{d}{2}\left( \breve{\Delta}_{i+1}^{n} -2 \breve{\Delta}_{i}^{n} + \breve{\Delta}_{i-1}^{n}  \right) 
+ \frac{r}{2}\breve{\Delta}_{i}^{n} 
\right] 
\notag \\
+r\left[ q_i^{n}-\frac{c}{4}\left( q_{i+1}^n-q_{i-1}^{n}\right)+\frac{d}{2}\breve{\Delta}_{i}^{n} +\frac{r}{2} q_i^{n} \right]
\label{eq:local_ader_alphacnst}
\end{gather}}
	where $d=\frac{\alpha\Delta t}{\Delta x^2}$ { and $\breve{\Delta}_{i}^{n} = q_{i+1}^{n}-2q_{i}^{n}+q_{i-1}^{n} $}.
\end{remark2}

\begin{remark2}
	There exist $c_{M}, \, d_{M},\, r_{m}\in \mathbb{R}$ such that the LADER scheme,
	\eqref{eq:local_ader_alphacnst}, is stable in the 4-orthotopes
	\begin{align} O_{c_{M}, d_{M}, r_{m}}=\left\lbrace (\theta,c,r,d)\; |\; \theta\in[-\pi,\pi],\; c\in [0,c_{M}],\; d \in[0,d_{M}], \;\right.\notag\\ \left.  r \in[r_{m},0],\; c_{M},\, d_{M}\in\mathbb{R}^{+}, \; r_{m}\in\mathbb{R}^{-}\right\rbrace. \label{eq:rc_cdr}\end{align}
	To represent a feasible 4-orthotope we consider the isosurface of level one of 
	\begin{equation}
	m_{\theta}(c,d,r)=\max_{\theta\in [-\pi,\pi]} \left\| A(\theta,c,d,r)\right\| \label{eq:amplification_mtheta}
	\end{equation}
	where $A(\theta,c,d,r)$ is the function of the amplification factor of the scheme (see \cite{BTVC16}). 
	In Figure \ref{fig:stability_localader} we can observe that the 4-orthotope defined by
	$c_{M}=0.3$, $d_{M}=0.2$ and $r_{m}=-0.5$ is embedded in the stability region.
	\begin{figure}
		\centering
		\includegraphics[width=0.5\linewidth]{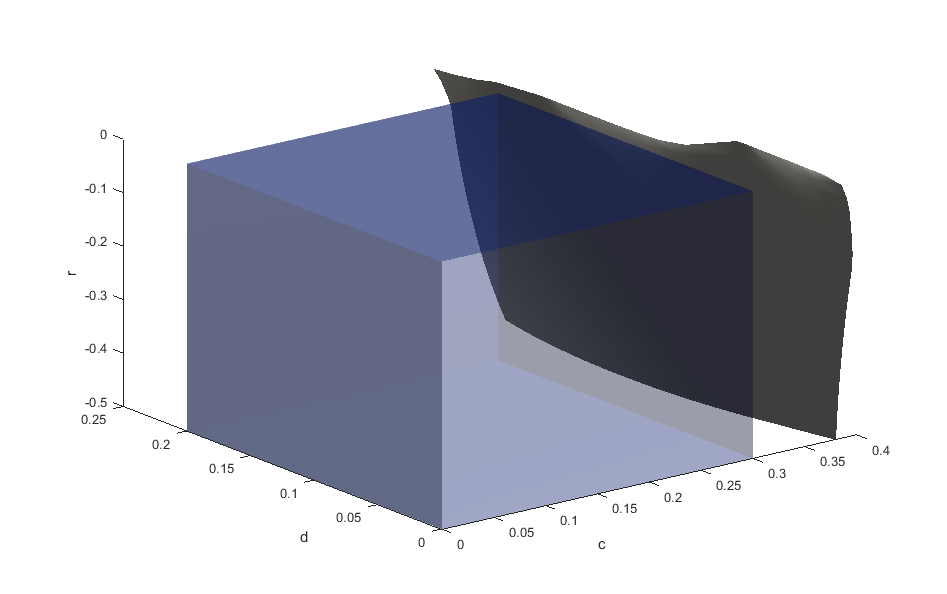}\hfill
		\includegraphics[width=0.5\linewidth]{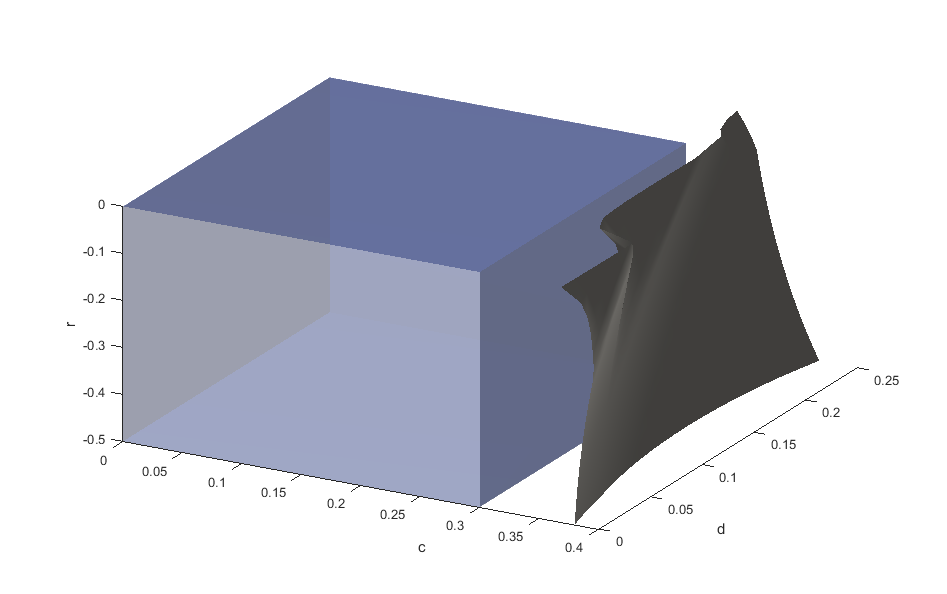}
		\caption{Two different views of the isosurface of level one of function $m_{\theta}$, \eqref{eq:amplification_mtheta},
			(grey) and the 4-orthotope of stability $\mathcal{O}_{0.3,0.2,-0.5}$ for the linear advection-diffusion-reaction equation (blue).}
		\label{fig:stability_localader}
	\end{figure}
\end{remark2}

\begin{lema}
	LADER scheme, \eqref{eq:local_ader3}, is of second order in time and space.	
\end{lema}
\begin{proof}
	To prove the accuracy of Scheme \eqref{eq:local_ader3}, we recall the analysis 
	carried out for the ADER scheme introduced in \cite{BTVC16} 
	and we detail the terms which have changed:
	
	\begin{itemize}
		\item Local truncation error contribution of the flux term neglecting the diffusion term contribution:
	\end{itemize}
	\begin{align}
	&&& \frac{\lambda}{\Delta x} \left[ q(x_{{ i}},t^{n}) -(x_{{ i}-1},t^{n}) 
	+\frac{1}{2} \left(  q(x_{{ i}},t^{n}) - 2 q(x_{{ i}-1},t^{n}) + q(x_{{ i}-2},t^{n})\right) 
	\right. \notag\\&&- & \left.
	\frac{c}{2} \left( q(x_{{ i}+1},t^{n}) - 2 q(x_{{ i}},t^{n}) + q(x_{{ i}-1},t^{n})\right)\right]
	\notag\\&=& &
	\frac{\lambda}{\Delta x} \left\lbrace 
	\partial_{x} q(x_{{ i}},t^{n}) \Delta x - \frac{1}{2}\partial_{x}^{(2)} q(x_{{ i}},t^{n})\Delta x^{2} +
	\frac{1}{6}\partial_{x}^{(3)} q(x_{{ i}},t^{n})\Delta x^{3}+\mathcal{O}\left(\Delta x^{4}\right) \right. \notag\\&&+&\left.
	\frac{1}{2}  \left[ 
	\partial_{x}^{(2)} q(x_{{ i}},t^{n})\Delta x^{2}- \partial_{x}^{(3)} q(x_{{ i}},t^{n})\Delta x^{3}+\mathcal{O}\left(\Delta x^{4}\right)
	\right]
	\right. \notag\\&&-& \left.  \frac{c}{2}\left[\partial_{x}^{(2)} q(x_{{ i}},t^{n})\Delta x^{2}-
	\frac{1}{12}\partial_{x}^{(4)} q(x_{{ i}},t^{n})\Delta x^{4}+\mathcal{O}\left(\Delta x^{5}\right)
	\right] 
	\right\rbrace\notag\\
	&=& &\lambda\partial_{x} q(x_{{ i}},t^{n})- \frac{\lambda^{2}\Delta t}{2} \partial_{x}^{(2)} q(x_{{ i}},t^{n}) +\mathcal{O}\left(\Delta x^{2}\right)
	\end{align}
	
	\begin{itemize}
		\item Local truncation error contribution of the diffusion term to the flux term:
	\end{itemize}
	\begin{align}
	& & & \frac{\lambda \Delta t}{2 \Delta x^2} \left\lbrace\left[ 
	\alpha(x_{{ i}+1},t^{n})\frac{q(x_{{ i}+2},t^{n})-q(x_{{ i}+1},t^{n}) }{\Delta x}
	- \alpha(x_{{ i}},t^{n})\frac{q(x_{{ i}},t^{n})-q(x_{{ i}-1},t^{n}) }{\Delta x} \right]
	\right. \notag\\& & -&\left.\left[ 
	\alpha(x_{{ i}},t^{n})\frac{q(x_{{ i}+1},t^{n})-q(x_{{ i}},t^{n}) }{\Delta x}
	- \alpha(x_{{ i}-1},t^{n})\frac{q(x_{{ i}-1},t^{n})-q(x_{{ i}-2},t^{n}) }{\Delta x} \right]
	\right\rbrace\notag\\
	& = & & \frac{\lambda \Delta t}{2 \Delta x^3} \left\lbrace
	\alpha(x_{{ i}+1},t^{n})\left[   \partial_x q(x_{{ i}},t^{n}) \Delta x
	+\frac{3}{2}\partial_{x}^{(2)} q(x_{{ i}},t^{n}) \Delta x^2\right.\right.\notag\\
	& & + &\left. \left.\left.  \frac{7}{6}\partial^{(3)}_x q(x_{{ i}},t^{n}) \Delta x^3
	+\mathcal{O}(\Delta x^4)\right] - \alpha(x_{{ i}},t^{n})\left[\phantom{\frac{b}{b}}\!\!\!
	\right.\right.\right.\notag\\ & & &\left. \left. 2\partial_x q(x_{{ i}},t^{n}) \Delta x
	+\frac{1}{3}\partial^{(3)}_x q(x_{{ i}},t^{n}) \Delta x^3+\mathcal{O}(\Delta x^4) )  \right]
	\right. \notag\\ & & +&\left.
	\alpha(x_{{ i}-1},t^{n})\left[ \partial_x q(x_{{ i}},t^{n}) \Delta x
	-\frac{3}{2}\partial^{(2)} q(x_{{ i}},t^{n}) \Delta x^2 
	\right.\right.\notag\\ & & + &\left.\left.\left.
	\frac{7}{6}\partial^{(3)} q(x_{{ i}},t^{n}) \Delta x^3+\mathcal{O}(\Delta x^4)\right]  \right.\right.\notag\\ 
	& = & & \frac{\lambda \Delta t}{2\Delta x^3} \left[ \partial_x^{(2)}\alpha(x_{{ i}},t^{n})
	\partial_x q(x_{{ i}},t^{n}) \Delta x^3 + 3\partial_x \alpha_{i}^{n}\partial_x^{(2)}
	q(x_{{ i}},t^{n}) \Delta x^3\right.\notag\\& & +&\left.
	2 \alpha_{i}^{n}\partial_x^{(3)}q(x_{{ i}},t^{n}) \Delta x^3 +\mathcal{O}\left(\Delta x^4\right)\right]\notag\\
	& = & & \frac{\lambda \Delta t}{2}\left[
	\partial_{x}^{(2)}\left(  \alpha(x_{{ i}},t^{n})\partial_x q(x_{{ i}},t^{n})\right)
	+ \partial_x\left( \alpha(x_{ i},t^n)\partial_x^{(2)} q(x_{ i},t^n) \right)
	\right]\notag\\
	& & +& \mathcal{O}(\Delta x \Delta t)
	\end{align}
	
	\begin{itemize}  
		\item Local truncation error contribution of the diffusion term:
	\end{itemize}
	\begin{align}
	& & - &\frac{1}{\Delta x^2} \left\lbrace 
	\overline{ \alpha}(x_{{ i}+\frac{1}{2}},t^n)   \left[ \phantom{\frac{b}{b}}\!\!\!
	q(x_{{ i}+1},t^n)-q(x_{{ i}},t^{n}) \right.\right.\notag\\ & & +& \left. \left.
	+\frac{\Delta t}{2}\left( \phantom{\frac{b}{b}}\!\!\! \frac{1}{\Delta x^2}\alpha(x_{{ i}
		+\frac{3}{2}},t^n) \left[  q(x_{{ i}+2},t^{n})-q(x_{{ i}+1},t^{n})\right] \right. \right.\right.\notag\\	& &
	- & \left.\left.\left. \frac{1}{\Delta x^2}\alpha(x_{{ i}+\frac{1}{2}},t^n)\left[  2q(x_{{ i}+1},t^{n})-2q(x_{{ i}},t^{n})\right]
	\right.\right.\right. \notag\\ 
	& &+& \left.\left.\left. \frac{1}{\Delta x^2}\alpha(x_{{ i}-\frac{1}{2}},t^n)\left[  q(x_{{ i}},t^{n})-q(x_{{ i}-1},t^{n})\right]
	\right.\right.\right. \notag\\   
	& &+ & \left.\left.\left.  \beta\left[ q(x_{{ i}+1},t^{n})-q(x_{{ i}},t^{n})\right] \phantom{\frac{b}{b}}\!\!\!\right) 
	\right] 
	\right.\notag\\ & & + &\left.
	\overline{\alpha}(x_{{ i}-\frac{1}{2}},t^n)  \left[\phantom{\frac{b}{b}}\!\!\!\
	q(x_{{ i}-1},t^{n})-q(x_{{ i}},t^{n})
	\right. \right.\notag\\ & & +& \left. \left.
	\frac{\Delta t}{2}\left( \phantom{\frac{b}{b}}\!\!\!\! -\frac{1}{\Delta x^2}\alpha(x_{{ i}+\frac{1}{2}},t^n)
	\left[  q(x_{{ i}+1},t^{n})-q(x_{{ i}},t^{n})\right]  \right. \right.\right.\notag\\
	& & + & \left. \left. \left. 
	\frac{1}{\Delta x^2}\alpha(x_{{ i}-\frac{1}{2}},t^n)\left[  2q(x_{{ i}},t^{n})-2q(x_{{ i}-1},t^{n})\right]
	\right.\right.\right. \notag\\ & & -&\left.\left.\left.
	\frac{1}{\Delta x^2}\alpha(x_{{ i}-\frac{3}{2}},t^n)\left[  q(x_{{ i}-1},t^{n})-q(x_{{ i}-2},t^{n})\right]
	\right.\right.\right. \notag\\ & & +&\left.\left.\left.
	\beta\left[ q(x_{{ i}-1},t^{n})-q(x_{{ i}},t^{n})\right] \phantom{\frac{b}{b}}\!\!\!\right) 
	\right] 
	\right\rbrace\notag\\
	&=& - &\partial_x\left(\alpha(x_{ i},t^n)\partial_x q(x_{ i},t^n)\right) + \frac{\Delta t}{2}\left\lbrace
	\phantom{\frac{b}{b}}\!\!\! \partial_x\left(\partial_t\alpha(x_{ i},t^n)\partial_x q(x_{ i},t^n)\right) \right. \notag\\& & +&\left.
	\partial_x\left[\alpha(x_{ i},t^n)\partial^{(2)}_x\left(  \alpha(x_{ i},t^n) \partial_xq(x_{ i},t^n)\right) \right]
	\right. \notag\\ & & +&\left.
	\beta \partial_x\left(\alpha(x_{ i},t^n)\partial_x q(x_{ i},t^n)\right) \phantom{\frac{b}{b}}\!\!\!
	\right\rbrace 
	+\mathcal{O}(\Delta x^2 )+\mathcal{O}(\Delta x \Delta t ).
	\end{align}
	
	Finally, taking into account the truncation error of the remaining terms of the scheme and Cauchy-Kovalevskaya equality, we get
	\begin{align}
	\tau^n& =  & & \partial_t q(x_{ i},t^n) +\lambda \partial_x q(x_{ i},t^n) -\partial_x\left[
	\alpha(x_{ i},t^n)\partial_x q(x_{ i},t^n)\right] - \beta q(x_{ i},t^n) 
	\notag\\	& & +&
	\frac{\Delta t}{2} \left\lbrace\phantom{\frac{b}{b}}\!\!\! \partial^{(2)}_t q(x_{ i},t^n) +\lambda
	\partial_x \left[  -\lambda  \partial_x q(x_{ i},t^n) + \beta q(x_{ i},t^n)  \right] \right.\notag\\
	& & -&\left. \beta\left[  -\lambda\partial_x q(x_{ i},t^n)  +\beta q(x_{ i},t^n) \right]  
	- \partial_x\left[\partial_t\alpha(x_{ i},t^n)\partial_x q(x_{ i},t^n)\right] \right. \notag\\& &+&\left.
	\lambda \partial_x\left[\alpha(x_{ i},t^n)\partial^{(2)}_x q(x_{ i},t^n)\right]
	- \partial_x\left\lbrace\alpha(x_{ i},t^n)\partial^{(2)}_x\left[  \alpha(x_{ i},t^n) \partial_xq(x_{ i},t^n)\right] \right\rbrace
	\right. \notag\\ & &-&\left.
	\beta \partial_x\left[\alpha(x_{ i},t^n)\partial_x q(x_{ i},t^n)\right]
	-\beta\partial_x\left[ \alpha(x_{ i},t^n)\partial_x q(x_{ i},t^n)\right]  \right.  \notag\\
	& & + &\left.
	\partial^{(2)}_x\left[\alpha(x_{ i},t^n)\partial_x q(x_{ i},t^n)\right] \phantom{\frac{b}{b}}\!\!\!\right\rbrace
	 + \mathcal{O}\left(\Delta t^2 \right) +\mathcal{O}\left(\Delta x^2 \right) +\mathcal{O}\left(\Delta x \Delta t \right)
	\notag\\	& =  & &\mathcal{O}\left(\Delta t^2 \right) +\mathcal{O}\left(\Delta x^2 \right)+\mathcal{O}(\Delta x \Delta t ).
	\end{align}
\end{proof}

{
\section{Manufactured tests.  Source terms}\label{sec:appendix_sourceterms}
In this appendix we describe the source terms used in the  manufactured test (see Sections \ref{sec:pb_mms_laminar} and \ref{sec:manufacturedtest2}): 

\begin{itemize}
	\item Manufactured test 1. Laminar flow.
\end{itemize}
{\footnotesize 
	\begin{eqnarray}
\mathbf{f}_{\mathbf{u}_1}(x,y,z,t)&=&
\pi y \cos(\pi t y) \cos(\pi t z) - \pi t \sin(\pi t (x + y + z)) \notag\\& &- \pi z \sin(\pi t y) \sin(\pi t z) 
+ 2 \pi^2 t^2 \mu \sin(\pi t y) \cos(\pi t z) \notag\\& &- \pi t \cos(\pi t z^3) \cos(\pi t y) \cos(\pi t z) \notag\\
& &- \pi t \sin(\pi t y) \sin(\pi t z) \exp(-2 \pi t^2 x),\\
\mathbf{f}_{\mathbf{u}_2}(x,y,z,t)&=&
\pi z^3 \sin(\pi t z^3) - \pi t \sin(\pi t (x + y + z)) - 6 \pi t z \mu \sin(\pi t z^3) \notag\\
& &- 9 \pi^2 t^2 z^4 \mu \cos(\pi t z^3) + 3 \pi t z^2 \exp(-2 \pi t^2 x) \sin(\pi t z^3),\\
\mathbf{f}_{\mathbf{u}_3}(x,y,z,t)&=&
- \pi t \sin(\pi t (x + y + z)) - 4 \pi^2 t^4 \mu \exp(-2 \pi t^2 x) \notag\\& &- 4 \pi t x \exp(-2 \pi t^2 x) 
\notag\\& &
- 2 \pi t^2 \sin(\pi t y) \exp(-2 \pi t^2 x) \cos(\pi t z).
\end{eqnarray}}

\begin{itemize}
\item Manufactured test 2. Turbulent flow with species transport.
\end{itemize}
{\footnotesize 
\begin{eqnarray}
\mathbf{f}_{\mathbf{u}_1}(x,y,z,t)&=&(2 \pi t  \cos(\pi t  x))/3 -  \pi t  \sin( \pi t ( x +  y +  z)) +  \pi  y  \cos( \pi t  y)  \cos( \pi t  z) \notag\\
&&-  \pi  z  \sin( \pi t  y)  \sin( \pi t  z) + 2  \pi^2 t^2  \mu  \sin( \pi t  y)  \cos( \pi t  z) \notag\\
&&-  \pi t  \cos( \pi t  z^3)  \cos( \pi t  y)  \cos( \pi t  z) -  \pi t  \sin( \pi t  y)  \sin( \pi t  z) \exp(-2  \pi t^2  x) \notag\\
&&+ (2  \pi^2 C_{\mu} t^2  \sin( \pi t  y)  \cos( \pi t  z) ( \sin( \pi t  x) + 2)^2)/(\exp(- \pi t  z) + 1) \notag\\
&&+ ( \pi^2 C_{\mu} t^2  \sin( \pi t  y)  \sin( \pi t  z) \exp(- \pi t  z)
\notag\\ &&
( \sin( \pi t  x) + 2)^2)/(\exp(- \pi t  z) + 1)^2,\\
\mathbf{f}_{\mathbf{u}_2}(x,y,z,t)&=& \pi  z^3  \sin( \pi t  z^3) -  \pi t  \sin( \pi t ( x +  y +  z)) - 6  \pi t  z  \mu  \sin( \pi t  z^3) \notag\\
&& 9  \pi^2 t^2  z^4  \mu  \cos( \pi t  z^3) + 3  \pi t  z^2 \exp(-2  \pi t^2  x)  \sin( \pi t  z^3) \notag\\
&&- (9  \pi^2 C_{\mu} t^2  z^4  \cos( \pi t  z^3) ( \sin( \pi t  x) + 2)^2)/(\exp(- \pi t  z) + 1) \notag\\
&&- (6  \pi C_{\mu} t  z  \sin( \pi t  z^3) ( \sin( \pi t  x) + 2)^2)/(\exp(- \pi t  z) + 1) \notag\\
&&- (3  \pi^2 C_{\mu} t^2  z^2 \exp(- \pi t  z)  \sin( \pi t  z^3)\notag\\
&& ( \sin( \pi t  x) + 2)^2)/(\exp(- \pi t  z) + 1)^2,\\[6pt]
\mathbf{f}_{\mathbf{u}_3}(x,y,z,t)&=&(4  \pi^2 C_{\mu} t^3 \exp(-2  \pi t^2  x)  \cos( \pi t  x) ( \sin( \pi t  x) + 2))/(\exp(- \pi t  z) + 1) \notag\\
&&- 4  \pi^2 t^4  \mu \exp(-2  \pi t^2  x) - 4  \pi t  x \exp(-2  \pi t^2  x) \notag\\
&&- 2  \pi t^2  \sin( \pi t  y) \exp(-2  \pi t^2  x)  \cos( \pi t  z) \notag\\
&&- (4  \pi^2 C_{\mu} t^4 \exp(-2  \pi t^2  x) ( \sin( \pi t  x) + 2)^2)/(\exp(- \pi t  z) + 1) \notag\\
&&-  \pi t  \sin( \pi t ( x +  y +  z)),\\[6pt]
f_{k}(x,y,z,t)&=& \exp(- \pi t  z) +  \pi  x  \cos( \pi t  x) \notag\\
&&- (C_{\mu} ( \sin( \pi t  x) + 2)^2 (2  \pi t^2 \exp(-2  \pi t^2  x) \notag\\
&&+  \pi t  \sin( \pi t  y)  \sin( \pi t  z))^2)/(\exp(- \pi t  z) + 1) \notag\\
&&+  \pi^2 t^2  \mu  \sin( \pi t  x) +  \pi t  \sin( \pi t  y)  \cos( \pi t  x)  \cos( \pi t  z) \notag\\
&&- (9  \pi^2 C_{\mu} t^2  z^4  \sin( \pi t  z^3)^2 ( \sin( \pi t  x) + 2)^2)/(\exp(- \pi t  z) + 1) \notag\\
&&- ( \pi^2 C_{\mu} t^2  \cos( \pi t  y)^2  \cos( \pi t  z)^2 ( \sin( \pi t  x) + 2)^2)/(\exp(- \pi t  z) + 1) \notag\\
&&- (2  \pi^2 C_{\mu} t^2  \cos( \pi t  x)^2 ( \sin( \pi t  x) + 2))/(\sigma_{k} (\exp(- \pi t  z) + 1)) \notag\\
&&+ ( \pi^2 C_{\mu} t^2  \sin( \pi t  x) ( \sin( \pi t  x) + 2)^2)/(\sigma_{k} (\exp(- \pi t  z) + 1)) + 1,\\[6pt]
f_{\varepsilon}(x,y,z,t)&=& 
(C_{2\,\varepsilon} (\exp(- \pi t  z) + 1)^2)/( \sin( \pi t  x) + 2) -  \pi  z \exp(- \pi t  z) \notag\\
&&-  \pi^2 t^2  \mu \exp(- \pi t  z) \notag\\
&&- (C_{1\,\varepsilon} (\exp(- \pi t  z) + 1) ((C_{\mu} ( \sin( \pi t  x) + 2)^2 (2  \pi t^2 \exp(-2  \pi t^2  x) \notag\\
&&+  \pi t  \sin( \pi t  y)  \sin( \pi t  z))^2)/(\exp(- \pi t  z) + 1) \notag\\
&&+ (9  \pi^2 C_{\mu} t^2  z^4  \sin( \pi t  z^3)^2 ( \sin( \pi t  x) + 2)^2)/(\exp(- \pi t  z) + 1) \notag\\
&&+ ( \pi^2 C_{\mu} t^2  \cos( \pi t  y)^2  \cos( \pi t  z)^2 ( \sin( \pi t  x) \notag\\
&&+ 2)^2)/(\exp(- \pi t  z) + 1)))/( \sin( \pi t  x) + 2) \notag\\
&&-  \pi t \exp(- \pi t  z) \exp(-2  \pi t^2  x) \notag\\
&&- ( \pi^2 C_{\mu} t^2 \exp(- \pi t  z) ( \sin( \pi t  x) + 2)^2)/(\sigma_{\varepsilon} (\exp(- \pi t  z) + 1)) \notag\\
&&+ ( \pi^2 C_{\mu} t^2 \exp(-2  \pi t  z) ( \sin( \pi t  x) + 2)^2)/(\sigma_{\varepsilon} (\exp(- \pi t  z) + 1)^2),\\[6pt]
f_{y}(x,y,z,t)&=&  \pi  x  \cos( \pi t  x) +  \pi^2 t^2  \sin( \pi t  x) (\mathcal{D} + (C_{\mu} ( \sin( \pi t  x) \notag\\
&&+ 2)^2)/(Sc_{t} (\exp(- \pi t  z) + 1))) \notag\\
&&+  \pi t  \sin( \pi t  y)  \cos( \pi t  x)  \cos( \pi t  z)\notag\\
&& - (2  \pi^2 C_{\mu} t^2  \cos( \pi t  x)^2 ( \sin( \pi t  x) + 2))/(Sc_{t} (\exp(- \pi t  z) + 1)).
\end{eqnarray}} }

\end{document}